\documentclass{svproc}
\usepackage{amsfonts}
\usepackage{graphicx}
\usepackage{amsmath,amsxtra,latexsym, amssymb, amscd, pb-diagram,bezier}
\usepackage{color,tikz}
\usetikzlibrary{decorations.pathreplacing,calc}
\usepackage{rotating}
\tikzset{liltext/.style={font=\tiny}}
\usepackage{array}
\usepackage{float}
\usepackage{verbatim}
\usepackage{enumerate}
\renewcommand{\quote}[1]{``#1''}
\newcommand{\beq}{\begin{equation}}
\newcommand{\eeq}{\end{equation}}
\newcommand{\beqst}{\begin{equation*}}
\newcommand{\eeqst}{\end{equation*}}
\newcommand{\la}{\lambda}
\newcommand{\al}{\alpha}
\newcommand{\ba}{\beta}
\newcommand{\si}{\sigma}
\newcommand{\de}{\delta}
\newcommand{\ga}{\gamma}
\newcommand{\om}{\omega}
\newcommand{\ta}{\tau}

\newcommand{\ve}{\varepsilon}

\newcommand{\La}{\Lambda}
\newcommand{\Om}{\Omega}
\newcommand{\Th}{\Theta}
\newcommand\lap{\Delta}

\newcommand{\pa}{\partial}
\newcommand{\iy}{\infty}
\newcommand{\ra}{\rightarrow}

\newcommand{\R}{\mathbb{R}}

\DeclareMathOperator{\diag}{diag}
\DeclareMathOperator{\supp}{supp}

\newcommand\Zhyp{Z_{{hyp}}}
\newcommand\Zosc{Z_{{osc}}}

\newcommand\Zred{Z_{{red}}}
\newcommand\Zell{Z_{{ell}}}
\newcommand\Zdiss{Z_{{diss}}}
\newcommand{\jpnxi}{\langle\xi\rangle_{\la(t),\om(t)}}
\newcommand{\jpn}{\langle\xi\rangle}

\newcommand\D{{\mathcal D}}

\newcommand{\smallO}[1]{\mathop{}\!o{\big(#1\big)}}

\def\fu{\hat u}

\def\om{\omega}
\def\ta{\tau}
\def\D{\mathcal{D}}
\def\R{\mathcal{R}}
\def\Q{\mathcal{Q}}
\def\E{\mathcal{E}}
\def\th{\theta}
\def\Om{\Omega}
\usepackage{url}

\begin{document}
\mainmatter              
\title{The influence of oscillations on energy estimates for damped wave models with time-dependent propagation speed and dissipation}
\titlerunning{The influence of oscillations on energy estimates for damped wave models}  
%
\author{Halit Sevki Aslan$^1$ \and Michael Reissig$^{*1}$}
\authorrunning{H. S. Aslan and M. Reissig} 
\institute{$ ^{1}$ Faculty for Mathematics and Computer Science, Technical University Bergakademie Freiberg, Pr\"uferstr 9 - 09596 Freiberg - Germany \\
\email{$ ^{*}$reissig@math.tu-freiberg.de} }
\maketitle              

\begin{abstract}
The aim of this paper is to derive higher order energy estimates for solutions to the Cauchy problem
for damped wave models with time-dependent propagation speed and dissipation. The model of interest is
\begin{align*}
\begin{cases}
u_{tt}-\la^2(t)\om^2(t)\lap u +\rho(t)\om(t)u_t=0, & (t,x)\in[0,\infty)\times \mathbb{R}^n, \\
u(0,x)=u_0(x), \,\,\,\, u_t(0,x)=u_1(x), & x\in\mathbb{R}^n,
\end{cases}
\end{align*}
The coefficients $\la=\la(t)$ and $\rho=\rho(t)$ are shape functions and $\om=\om(t)$ is a bounded oscillating
function. If $\om(t)\equiv1$ and $\rho(t)u_t$ is an effective dissipation term, then $L^2-L^2$ energy
estimates are proved in \cite{BuiReiBook}. In contrast, the main goal of the present paper is to generalize
the previous results to coefficients including an oscillating function in the time-dependent coefficients.
We will explain how the interplay between the shape functions and oscillating behavior of the coefficient $\om=\om(t)$
will influence energy estimates.
\end{abstract}
\section{Introduction}
Let us consider the following Cauchy problem for the damped wave equation with time-dependent propagation speed and dissipation:
\begin{align} \label{Cauchy.Prob.Bui}
\begin{cases}
u_{tt}-a^2(t)\lap u+b(t)u_t=0, & (t,x)\in[0,\infty)\times \mathbb{R}^n, \\
u(0,x)=u_0(x), \,\,\, u_t(0,x)=u_1(x), & x\in\mathbb{R}^n.
\end{cases}
\end{align}
The total energy of the solution to \eqref{Cauchy.Prob.Bui} is defined by
\[ \mathbb{E}(u,t)=\frac{1}{2}\int_{\mathbb{R}^n}\big( a^2(t)|\nabla u(t,x)|^2+|u_t(t,x)|^2 \big)dx. \]
The properties of the time-dependent propagation speed $a=a(t)$ and the coefficient $b=b(t)$ in the damping term have different effects on the behavior of $\mathbb{E}(u,t)$ as $t\ra\iy$. \\
First we consider the properties of the energy in the case  $b(t)\equiv0$ in \eqref{Cauchy.Prob.Bui}. If $0<a_0\leq a(t)\leq a_1$, in general one cannot expect the so-called \textit{generalized energy conservation} property (GEC), that is, the both sided estimate
\[ C_0\mathbb{E}(u,0)\leq\mathbb{E}(u,t)\leq C_1\mathbb{E}(u,0) \]
holds for the wave energy with positive constants $C_0$ and $C_1$. The main reason is, that an oscillating behavior of $a=a(t)$ may have a very deteriorating influence on the energy behavior of the solutions to \eqref{Cauchy.Prob.Bui} (see \cite{Col.06}, \cite{ReiYag00}). However, if we assume for the oscillating behavior of $a=a(t)$ the assumptions
\[ |a^{(k)}(t)|\leq C_k(1+t)^{-k} \,\,\, \text{for} \,\,\, k=0,1,2, \]
then (GEC) holds (see \cite{ReiSmi05}). In this case only very slow oscillations are allowed. In general, very fast oscillating  coefficients one might expect to destroy the estimates, which are valid  for very slow oscillating  coefficients. However, in \cite{Hir07} even though the oscillations are very fast, the author has proved (GEC) to \eqref{Cauchy.Prob.Bui} under the following additional assumptions to the coefficient:
\begin{equation}
|a^{(k)}(t)|\leq C_k(1+t)^{-kr} \,\,\, \text{for} \,\,\, k=0,1,\cdots,M,
\end{equation}
and
\begin{equation} \label{Stab.Cond.Hir.}
\int_{0}^{t}|a(\ta)-a_{\iy}|d\ta\leq C(1+t)^q,
\end{equation}
for some real $a_\iy$ and $q\in(0,1)$ with $r>q+\dfrac{1-q}{M}$. Here \eqref{Stab.Cond.Hir.} is called the \textit{stabilization condition} (see \cite{Hir07}) and by this condition one can get some benefit of higher order regularity of $a=a(t)$. Moreover, in the case  $b(t)\equiv0$ in the paper \cite{H-W09} the authors studied the Cauchy problem \eqref{Cauchy.Prob.Bui} after introducing $a(t)=\la(t)\om(t)$ with the monotonously increasing shape function $\la=\la(t)$ and the (bounded) oscillating function $\om=\om(t)$. By using the $\mathcal{C}^M$ property of $\la=\la(t)$ and $\om=\om(t)$ and the idea of stabilization condition they proved a two sided estimate
\[ C_0\leq \frac{1}{\la(t)}\mathbb{E}_\la(u,t)\leq C_1, \]
where the non-negative constants $C_0$ and $C_1$ depend on the data and
\[ \mathbb{E}_\la(u,t)=\frac{1}{2}\int_{\mathbb{R}^n}\big( \la^2(t)|\nabla u(t,x)|^2+|u_t(t,x)|^2 \big)dx. \]
Now we discuss the Cauchy problem \eqref{Cauchy.Prob.Bui} in the case of constant speed of propagation $a(t)\equiv1$. Assuming a suitable control of the oscillations in $b=b(t)$, the following classification of damping terms $b(t)u_t$ is proposed in \cite{Wirth.non-eff} and \cite{Wirth.eff.}: \textit{non-effective dissipation, effective dissipation, scattering producing} and \textit{over-damping producing}.
\par Finally, let us consider the Cauchy problem \eqref{Cauchy.Prob.Bui} with time-dependent propagation speed and dissipation. In \cite{BuiReiBook} the authors studied the Cauchy problem \eqref{Cauchy.Prob.Bui} assuming a suitable control of the oscillations of $a=a(t)$ and $b=b(t)$. They proposed a classification of the damping term $b(t)u_t$ in terms of an increasing speed of propagation $a=a(t)$, given in \cite{Wirth.non-eff} and \cite{Wirth.eff.}. For the effective dissipation case, the authors proved the energy of the solution to \eqref{Cauchy.Prob.Bui} satisfies the following $L^2-L^2$ estimate for all $t>0$:
\begin{align*}
\big\| u_t(t,\cdot),a(t)\nabla u(t,\cdot) \big\|_{L^2} \lesssim a(t)\Big( 1+\int_{0}^{t}\frac{a^2(\ta)}{b(\ta)}d\ta \Big)^{-\frac{1}{2}}\big( \|u_0\|_{H^1}+\|u_1\|_{L^2} \big).
\end{align*}
The main goal of this paper is to derive higher order energy estimates for solutions of the Cauchy problem \eqref{Cauchy.Prob.Bui} with time-dependent speed of propagation and time-depending damping term both having a time-dependent oscillation term. For this reason, we assume that the coefficients $a=a(t)$ and $b=b(t)$ in \eqref{Cauchy.Prob.Bui} can be represented by the following products:
\begin{equation*}
a(t)=\la(t)\om(t) \quad \text{and} \quad b(t)=\rho(t)\om(t).
\end{equation*}
Here the time-dependent functions $\la=\la(t)$, $\rho=\rho(t)$ and $\om=\om(t)$ are smooth and strictly positive functions. Thus, in this paper we devote ourselves to the following Cauchy problem:
\begin{align} \label{Equ.first.Cauchy}
\begin{cases}
u_{tt}-\la^2(t)\om^2(t)\lap u +\rho(t)\om(t)u_t=0, & (t,x)\in[0,\infty)\times \mathbb{R}^n, \\
u(0,x)=u_0(x), \,\,\,\, u_t(0,x)=u_1(x), & x\in\mathbb{R}^n,
\end{cases}
\end{align}
where $\la$ is a monotonously increasing nontrivial shape function in the propagation speed, $\rho$ is a nontrivial shape function in the damping term and $\om$ is a bounded oscillating function in both propagation speed and damping term. Throughout this paper, we restrict ourselves to ``effective-like'' (due to oscillating term $\om=\om(t)$) damping $\rho(t)\om(t)u_t$ in the sense of \cite{BuiReiBook} and \cite{Wirth.eff.}. Here effective means that the solution behaves like that of a corresponding heat equation.
\begin{remark}
We assume that the damped wave equation \eqref{Equ.first.Cauchy} has the same oscillating function $\om=\om(t)$ in both speed of the propagation and damping term. Only for this model we are able to develop a suitable approach basing on tools from WKB-analysis. In particular, the division of the extended phase space into regions and zones
(see Section \ref{Sec3}) with reasonable properties of the separating lines depends heavily on the coincidence of the oscillating part in the propagation speed and the dissipation term.
\end{remark}
Then, the main results of this paper are the following estimates for Sobolev solutions to the Cauchy problem \eqref{Equ.first.Cauchy}.
\begin{theorem} [Main Theorem] \label{Par.Dep.Final.Thm.with.s=0}
We assume that $\la=\la(t)$, $\rho=\rho(t)$ and $\om=\om(t)$ satisfy the conditions (A1) to (A5) and (B1) to (B6) (see below). Then, the Sobolev solutions to the Cauchy problem \eqref{Equ.first.Cauchy} satisfy the following estimates with $m\in[1,2)$ and $\sigma\geq 0$:
\begin{align*}
\|u(t,\cdot)\|_{\dot{H}^{\sigma}} &\lesssim \big( 1+B_\la(0,t) \big)^{-\frac{\sigma}{2}-\frac{n}{2}(\frac{1}{m}-\frac{1}{2})}\big( \|u_0\|_{L^m\cap H^{\sigma}}+\|u_1\|_{L^m\cap H^{[\sigma-1]_+}} \big), \\
\|u_t(t,\cdot)\| _{\dot{H}^{\sigma}} &\lesssim \la(t)
\big( 1+B_\la(0,t) \big)^{-\frac{\sigma}{2}-\frac{n}{2}(\frac{1}{m}-\frac{1}{2})-\frac{1}{2}}\big( \|u_0\|_{L^m\cap H^{\sigma+1}}+\|u_1\|_{L^m\cap H^{\sigma}} \big),
\end{align*}
where $B_\la(0,t):=\int_{0}^{t}\frac{\la^2(\ta)}{\rho(\ta)}d\ta$.
\end{theorem}
\begin{remark}
The estimates for $\|u(t,\cdot)\|_{\dot{H}^{\sigma}}$ from Theorem \ref{Par.Dep.Final.Thm.with.s=0} coincide with the ones for solutions to the Cauchy problem \eqref{Equ.first.Cauchy} with $\omega \equiv 1$. For this reason different influences of the conditions for $\omega$ on the estimates for solutions and derivatives (bad influence of the  oscillating behavior (A2) and improving influence of the stabilization condition (A3)) are in equilibrium. Some differences in the decay rate appear in the estimates for  $\|u_t(t,\cdot)\|_{\dot{H}^{\sigma}}$. Here we feel some bad influence of $\omega$.
\end{remark}
The content of the paper is organized as follows:
\begin{itemize}
\item In Section 2 we will introduce our assumptions for the coefficients $\la=\la(t)$, $\rho=\rho(t)$ and $\om=\om(t)$.
\item In Section 3 we apply (hyperbolic and elliptic) WKB-analysis to get representations of solutions. For this reason we divide the extended phase space into some zones and transform the second order equation to a system of first order for suitable micro-energies. This gives us precise information on the structure of fundamental solutions in different parts of the extended phase space.
\item Section 4 is devoted to glue the estimates for the fundamental solutions from different parts of the extended phase space.
\item In Section 5 we derive some estimates for large and small frequencies, which prove the main theorem of the Cauchy problem \eqref{Equ.first.Cauchy}.
\item In Section 6 we explain our assumptions and the main results for some examples of admissible coefficients.
\item Some concluding remarks and open problems in Section 7 complete the paper.
\end{itemize}
\paragraph{Notations} In this paper we use $f\lesssim g$ for positive functions $f$ and $g$ if there exists a constant $C>0$ such that the estimate $f\leq Cg$ is valid. Moreover, $f\approx g$ denotes if $f$ and $g$ satisfy$f\lesssim g$ and  $g\lesssim f$. $f=\smallO{g}$ denotes that $\lim\sup_{x\ra\iy}\big|\frac{f(x)}{g(x)}\big|=0$.  $(|\cdot|)$ denotes for a matrix the matrix of the absolute values of its entries. We introduce $[x]_+:=\max\{x,0\}$. Finally, by $H^\si$ and $\dot{H}^{\si}$, we denote Bessel and Riesz potential spaces based on $L^2$, respectively.
\section{Assumptions}
We assume the following conditions for $\la=\la(t)$ and $\om=\om(t)$ belonging to $\mathcal{C}^M(\mathbb{R}_+)$ with $M\geq 2$ (here we follow some ideas of \cite{H-W09}):
\begin{enumerate}
    \item [\textbf{(A1)}] $\la(t)>0$ and $\la'(t)>0$ for all times $t>0$ and the derivatives of $\la$ satisfy the conditions
    \[ \la_0\frac{\la(t)}{\La(t)}\leq\frac{\la'(t)}{\la(t)}\leq \la_1\frac{\la(t)}{\La(t)},\quad |d_t^k\la(t)|\leq \la_k\la(t)\Big( \frac{\la(t)}{\La(t)} \Big)^k, \,\,\, k=1,2,\cdots,M, \]
    where $\la_0$ and all $\la_k$ are positive constants and $\La(t)=1+\int_{0}^{t}\la(\ta)d\ta$ is a primitive of $\la(t)$;
    \item [\textbf{(A2)}] $0<c_0\leq\om(t)\leq c_1$ and the derivatives of $\om$ satisfy the conditions
    \[ |d_t^k\om(t)|\leq\om_k \Xi^{-k}(t), \,\,\, k=1,2,\cdots,M, \]
    where all $\om_k$ are positive constants and $\Xi=\Xi(t)$ is a positive, monotonous and continuous function satisfying the compatibility condition
    \[ C_1\Th(t)\leq\la(t)\Xi(t)\leq C_2\La(t),\]
    where $\Th=\Th(t)$ is a strictly increasing continuous function  with $\Th(0)=1$, $\Th(t)<\La(t)$ for $t>0$ and $\Th(t)=\smallO{\La(t)}$;
	\item [\textbf{(A3)}] $\om=\om(t)$ is $\la$-stabilizing towards $1$, that is,
	\[ \int_{0}^{t} \la(\ta)|\om(\ta)-1|d\ta \leq C_3\Th(t); \]
	\item [\textbf{(A4)}] for  $M \geq 2$ the following estimate holds:
	\[ \int_{t}^{\iy}\la^{-M}(\ta)\Xi^{-M-1}(\ta)d\ta\leq C_4\Th^{-M}(t); \]
    \item [\textbf{(A5)}] the function $F=F(\La(t))$ is defined by
    \[ \Xi(t)=\frac{F(\La(t))}{\la(t)\sqrt{F'(\La(t))}}. \]
    This implies that
    \[ \frac{1}{F\big( \La(t) \big)} = \int_{t}^{\iy} \la^{-1}(\ta)\Xi^{-2}(\ta)d\ta, \]
    where $F(\La(t))\ra\iy$ as $t\ra\iy$. Here we suppose that the right-hand side exists for all $t \geq 0$.
\end{enumerate}
Cauchy problems with the increasing speed of propagation have been considered in \cite{ReiYag00Lp-Lq}. Following their approach the assumption (A1) is standard. The assumption (A3), which is introduced as a \textit{stabilization condition}, allows us to control a certain amount of very fast oscillations. Especially, due to the deteriorating influence of oscillations on solutions by the aid of this stabilization condition and higher order regularity of the time-dependent coefficients may be compensated ``bad behavior'' of the very fast oscillations. For this reason, this stabilization condition describes an error made from the oscillating behavior of the coefficient $\om=\om(t)$.
\begin{remark}
If we consider very slow oscillations (according to the definitions in \cite{ReiYag00}), that is,  $\Th(t)\equiv\La(t)$ and $F(\La(t))\equiv\La(t)$ the assumptions (A3) to (A5) trivially hold and by these choices the stabilization condition disappears. Hence, the stabilization condition (A3) has a meaning only in the case $M\geq 2$.
\end{remark}
Now motivated by the considerations from \cite{Buithesis} and \cite{BuiReiBook}, in order to study the interaction between the shape functions $\la=\la(t)$, $\rho=\rho(t)$ and the oscillating function $\om=\om(t)$ we assume the following conditions:
\begin{enumerate}
	\item [\textbf{(B1)}] $\rho(t)>0, \quad \rho(t)=\mu(t)\dfrac{\la(t)}{\La(t)}$;
	\item [\textbf{(B2)}] $\big|d_t^k\mu(t)\big|\leq \mu_k\mu(t)\Big( \dfrac{\la(t)}{\La(t)} \Big)^k$ for $k=1,2,\cdots,M$,
    where all $\mu_k$ are
    positive constants;
	\item [\textbf{(B3)}] $\dfrac{\mu(t)}{\La(t)}$ is monotonic and $\mu(t)\ra\iy$ for $t\ra\iy;$
	\item [\textbf{(B4)}] $\dfrac{\la^2(t)}{\rho(t)}=\dfrac{\la(t)\La(t)}{\mu(t)}\notin L^1\mathbb{(R_+)};$
	\item [\textbf{(B5)}] $\big|\big( \rho(t)\om(t) \big)'\big|=\smallO{\left( \rho(t)\om(t) \right)^2}$ as $t\ra \iy$,
    this implies that $\mu(t)\dfrac{\Th(t)}{\La(t)}\ra \iy$ as $t\ra \iy$;
    \item [\textbf{(B6)}] $\displaystyle\int_{0}^{t}\frac{\la^2(\ta)}{\rho(\ta)}d\ta \leq C_5 F^2(\La(t))$,
    where $C_5$ is a positive constant.
\end{enumerate}
The assumptions (B3) and (B4) describe the effective damping case related to a given increasing propagation speed. In particular, (B4) excludes the over-damping case (see \cite{BuiReiBook}). The assumption (B5) allows us to control a certain amount of very fast oscillations in the damping term $\rho(t)\om(t)u_t$.
\section{Representation of solutions} \label{Sec3}
We perform the partial Fourier transformation $\fu(t,\xi)=F_{x\ra\xi}(u)(t,\xi)$ to \eqref{Equ.first.Cauchy} with respect to spatial variables. Then, we have
\begin{align} \label{Eq.first.Fourier}
\begin{cases}
\fu_{tt}+\la^2(t)\om^2(t)|\xi|^2\fu+\rho(t)\om(t)\fu_t=0, & (t,\xi)\in[0,\infty)\times \mathbb{R}^n, \\
\fu(0,\xi)=\fu_0(\xi), \quad \fu_t(0,\xi)=\fu_1(\xi),   & \xi\in\mathbb{R}^n.
\end{cases}
\end{align}
Applying the transformation
\begin{equation*}
\fu(t,\xi) = \exp\Big( -\frac{1}{2}\int_{0}^{t}\rho(\ta)\om(\ta)d\ta \Big)v(t,\xi),
\end{equation*}
transfers the Cauchy problem \eqref{Eq.first.Fourier} into
\begin{align} \label{after.dissipative}
\begin{cases}
v_{tt}+m(t,\xi)v=0, & (t,\xi)\in[0,\infty)\times \mathbb{R}^n, \\
v(0,\xi)=v_0(\xi), \,\,\, v_t(0,\xi)=v_1(\xi),   & \xi\in\mathbb{R}^n,
\end{cases}
\end{align}
where
\[ v_0(\xi)=\fu_0(\xi) \quad \mbox{and} \quad v_1(\xi)=\frac{\rho(0)\om(0)}{2}\fu_0(\xi)+\fu_1(\xi). \]
The coefficient $m=m(t,\xi)$ of the mass term is defined by
\begin{equation} \label{m(t,xi)}
m(t,\xi):=\la^2(t)\om^2(t)|\xi|^2-\frac{1}{4}\big( \rho(t)\om(t) \big)^2-\frac{1}{2}\big( \rho(t)\om(t) \big)'.
\end{equation}
Due to the assumptions (B2), (B3) and (B5) we see that $\big( \rho(t)\om(t) \big)'$ is a negligible term in \eqref{m(t,xi)}, that is, it holds $\big|\big(\rho(t)\om(t)\big)'\big|=\smallO{\left(\rho(t)\om(t)\right)^2}$ as $t\ra\iy$. Hence, we can introduce a separating curve $\Gamma$ as follows:
\[ \Gamma:=\Big\{ (t,\xi)\in\mathbb{R}_+\times\mathbb{R}^n : |\xi|=\frac{1}{2}\frac{\mu(t)}{\La(t)} \Big\}. \]
This curve divides the extended phase space into two regions, the hyperbolic region $\Pi_{hyp}$ and the elliptic region $\Pi_{ell}$, as follows:
\begin{align*}
\Pi_{hyp} = \bigg\{ (t,\xi) : |\xi|>\frac{1}{2}\frac{\mu(t)}{\La(t)} \bigg\} \quad \text{and} \quad \Pi_{ell} & = \bigg\{ (t,\xi): |\xi|<\frac{1}{2}\frac{\mu(t)}{\La(t)} \bigg\}.
\end{align*}
Let us define the auxiliary weight function
\begin{equation*}
\jpnxi := \sqrt{\Big|\la^2(t)\om^2(t)|\xi|^2 - \frac{\rho^2(t)\om^2(t)}{4}\Big|}.
\end{equation*}
\subsubsection{Division of the extended phase space.}
We divide both regions of the extended phase space into zones in order to organize the necessary steps of WKB-analysis. This procedure was developed in \cite{ReiYag00Lp-Lq} and \cite{Yagd.97}. Here we follow some ideas of \cite{BuiReiBook}. However, we shall restrict the considerations to a smaller hyperbolic zone and elliptic zone in the extended phase space to cope with stronger oscillations in $\om=\om(t)$ in our approach. The zones are defined as follows:
\begin{itemize}
\item hyperbolic zone:
\[ \Zhyp(N)=\Big\{ (t,\xi): \jpnxi\geq N\frac{\rho(t)\om(t)}{2} \Big\}\cap\Pi_{hyp}, \]
and $\Th(t)|\xi|\geq N$;
\item oscillation subzone:
\[ \Zosc(N,\ve)=\Big\{ (t,\xi): \ve\frac{\rho(t)\om(t)}{2}\leq \jpnxi \leq N\frac{\rho(t)\om(t)}{2} \Big\} \cap\Pi_{hyp}, \]
$\Th(t)|\xi|\leq N$ and $\La(t)|\xi|\geq N$;
\item reduced zone:
\[ \Zred(\ve) = \Big\{ (t,\xi): \jpnxi \leq \ve\frac{\rho(t)\om(t)}{2} \Big\}; \]
\item elliptic zone:
\[ \Zell(d_0,\ve) = \Big\{ (t,\xi): |\xi|\geq \frac{d_0}{F\big( \La(t) \big)}\Big\}\cap\Big\{ \jpnxi\geq\ve\dfrac{\rho(t)\om(t)}{2} \Big\}\cap\Pi_{ell}; \]
\item dissipative zone:
\[ \Zdiss(d_0)=\Big\{ (t,\xi): |\xi|\leq \frac{d_0}{F\big( \La(t) \big)} \Big\}\cap\Pi_{ell}. \]
\end{itemize}
Here in general, $N$ is a large positive constant and $\ve$ is a small positive constant. Both will be chosen later.
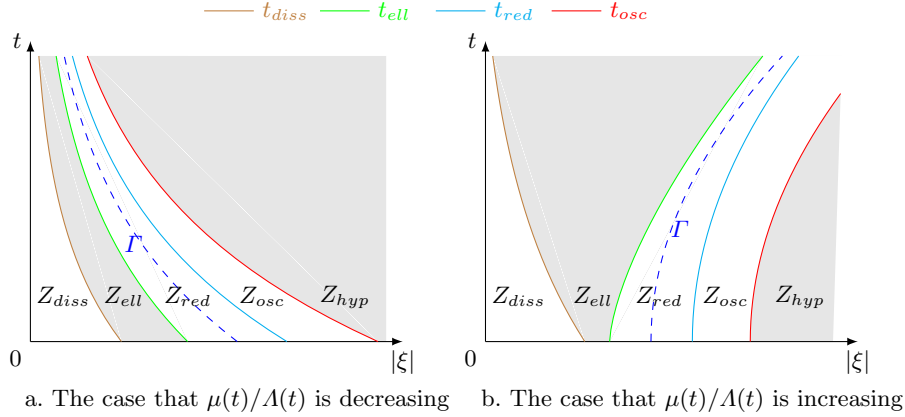
\begin{figure}[H]
	\begin{center}
		\begin{tikzpicture}[>=latex,xscale=1.1]
    \fill[domain=0:3.8, variable=\t,color=black!10!white] plot ({1.1*exp(-\t/1.6)},\t)--(1.1,0) -- (1.9,0)--(1.9, 3.8)--(0.1025,3.8)--cycle;
	\fill[domain=0:3.8,variable=\t,color=white] plot ({1.9*exp(-\t/2.1)},\t)--(1.9,0)--(4.2,0) -- (4.2,3.8) -- (0.311,3.8)--cycle;
	\fill[domain=0:3.8,variable=\t,color=black!10!white] plot ({4.2*exp(-\t/2.1)},\t)--(4.2,0)--(4.3,0) -- (4.3,3.8) -- (0.6875,3.8)--cycle;
	\draw[->] (0,0) -- (4.5,0)node[below]{$|\xi|$};
	\draw[->] (0,0) -- (0,4)node[left]{$t$};
    \node[below left] at(0,0){$0$};
    \draw[-, color=brown]  (2.1,4.4) -- (2.7,4.4)node[right]{$\textcolor{brown}{t_{diss}}$};
	\draw[-, color=green]  (3.5,4.4) -- (4.1,4.4)node[right]{$\textcolor{green}{t_{ell}}$};
    \draw[-, color=cyan]  (4.9,4.4) -- (5.5,4.4)node[right]{$\textcolor{cyan}{t_{red}}$};
    \draw[-, color=red]  (6.3,4.4) -- (6.9,4.4)node[right]{$\textcolor{red}{t_{osc}}$};
	\draw[dashed, domain=0:3.8, color=blue, variable=\t,] plot ({2.5*exp(-\t/2.1)},\t);
	\node[color=blue] at (1.25, 1.3){{\footnotesize $\Gamma$}};
	\draw[domain=0:3.8,color=red,variable=\t] plot ({4.2*exp(-\t/2.1)},\t);
	\node[color=black] at (3.8,0.6){{\footnotesize $Z_{hyp}$}};
	\draw[domain=0:3.8,color=cyan,variable=\t] plot ({3.1*exp(-\t/2.1)},\t);
	\node[color=black] at (2.8, 0.6){{\footnotesize $Z_{osc}$}};
	\node[color=black] at (1.9, 0.6){{\footnotesize $Z_{red}$}};
	\draw[domain=0:3.8,color=green,variable=\t] plot ({1.9*exp(-\t/2.1)},\t);
	\draw[domain=0:3.8,color=brown,variable=\t] plot ({1.1*exp(-\t/1.6)},\t);
	\node[color=black] at (1.13,0.6){{\footnotesize $Z_{ell}$}};
	\node[color=black] at (0.4,0.6){{\footnotesize $Z_{diss}$}};
	\node[below] at (2.5,-0.5) {{\footnotesize a. The case that $\mu(t)/\Lambda(t)$ is decreasing}};		
    \fill[domain=0:3.8,variable=\t,color=black!10!white] plot ({5.3+ 1.4*exp(-\t/2.4)},\t)--(6.7,0) -- (9.8,0)--(9.8, 3.8)--(5.5874,3.8)--cycle;
	\fill[domain=0:3.8,variable=\t,color=white] plot ({7 + 0.250014548*\t*sqrt(\t)},\t)--(7,0) -- (9.9,0)--(9.9, 3.8)--(8.85, 3.8)--cycle;	
	\fill[domain=0:3.3,variable=\t,color=black!10!white] plot ({8.7 + 0.100308642*pow(\t,2)},\t)--(9.7,0) -- (8.7,0)--cycle;
	\draw[->] (5.5,0) -- (10,0)node[below]{$|\xi|$};
	\draw[->] (5.5,0) -- (5.5,4)node[left]{$t$};
	\node[below left] at(5.5,0){$0$};
	\draw[dashed, domain=0:3.8, color=blue, variable=\t,] plot ({7.5 + 0.11*pow(\t,2)},\t);
	\node[color=blue] at (7.85, 1.5){{\footnotesize $\Gamma$}};
	\draw[domain=0:3.3,color=red,variable=\t] plot ({8.7 + 0.100308642*pow(\t,2)},\t);
	\node[color=black] at (9.3, 0.6){{\footnotesize $Z_{hyp}$}};
	\node[color=black] at (8.4, 0.6){{\footnotesize $Z_{osc}$}};
	\node[color=black] at (7.6, 0.6){{\footnotesize $Z_{red}$}};
	\draw[domain=0:3.8,color=cyan,variable=\t] plot ({8 + 0.089*pow(\t,2)},\t);
	\draw[domain=0:3.8,variable=\t,color=green] plot ({7 + 0.250014548*\t*sqrt(\t)},\t);
	\node[color=black] at (6.8,0.6){{\footnotesize $Z_{ell}$}};	
	\draw[domain=0:3.8,color=brown,variable=\t] plot ({5.3+ 1.4*exp(-\t/2.4)},\t);
	\node[color=black] at (5.9,0.6){{\footnotesize $Z_{diss}$}};
	(9.8, .3); \node[below] at
	(8,-.5) {{\footnotesize b. The case that $\mu(t)/\Lambda(t)$ is increasing}};
\end{tikzpicture}
\caption{Division of the extended phase space into zones}
\label{fig.zone}
\end{center}
\end{figure}
\vspace*{-2em}
Let us introduce the separating lines between the dissipative zone and the elliptic zone by $t_{diss}(|\xi|)=:t_{diss}$, between the elliptic zone and the reduced zone by $t_{ell}(|\xi|)=:t_{ell}$, between the reduced zone and oscillation subzone by $t_{red}(|\xi|)=:t_{red}$, between the oscillation subzone and the hyperbolic zone by $t_{osc}(|\xi|)=:t_{osc}$. By the definitions of the zones we can see that these separating lines really exist and can be described by functions due to the monotonicity of the functions $\frac{\mu(t)}{\La(t)}$ and $F(\La(t))$. \\
We define
\begin{align} \label{h_1(t,xi)}
h_1(t,\xi):=\chi\big( |\xi|F(\La(t)) \big)\frac{\la(t)}{F(\La(t))} + \Big( 1-\chi\big( |\xi|F(\La(t)) \big) \Big)\la(t)|\xi|
\end{align}
and
\begin{align} \label{h_2(t,xi)}
h_2(t,\xi):=\chi\bigg( \frac{\jpnxi}{\ve\frac{\rho(t)\om(t)}{2}} \bigg)\ve\frac{\rho(t)\om(t)}{2}+\bigg( 1-\chi\Big( \frac{\jpnxi}{\ve\frac{\rho(t)\om(t)}{2}} \Big) \bigg) \jpnxi,
\end{align}
where $\chi\in\mathcal{C}^{\iy}[0,\iy)$ such that $\chi(t)=1$ for $0\leq t\leq \frac{1}{2}$ and $\chi(t)=0$ for $t\geq 1$.\\
Introducing the micro-energy $U(t,\xi):=\big( h_1(t,\xi)\fu(t,\xi),D_t\fu(t,\xi) \big)^T$ we obtain from \eqref{Eq.first.Fourier} the system of first order
\begin{equation} \label{System.with.h_1(t,xi)}
    D_tU(t,\xi)=\underbrace{\left( \begin{array}{cc}
	\frac{D_t h_1(t,\xi)}{h_1(t,\xi)} & h_1(t,\xi) \\ [5pt]
	\frac{\la^2(t)\om^2(t)|\xi|^2}{h_1(t,\xi)} & i\rho(t)\om(t) \end{array} \right)}_{A(t,\xi)} U(t,\xi),
\end{equation}
with the initial condition $U(0,\xi) = \big( h_1(0,\xi)\fu(0,\xi),D_t\fu(0,\xi) \big)^T$. \\
On the other hand, we define the micro-energy $V(t,\xi):=\big( h_2(t,\xi)v(t,\xi),D_tv(t,\xi) \big)^T$. Then, by \eqref{after.dissipative} we obtain that $V=V(t,\xi)$ satisfies the following system of first order:
\begin{equation} \label{System.with.h_2(t,xi)}
D_tV(t,\xi)=\underbrace{\left( \begin{array}{cc}
\frac{D_t h_2(t,\xi)}{h_2(t,\xi)} & h_2(t,\xi) \\ [5pt]
\frac{m(t,\xi)}{h_2(t,\xi)} & 0 \end{array} \right)}_{A_V(t,\xi)} V(t,\xi),
\end{equation}
with the initial condition $V(0,\xi) = \big( h_2(0,\xi)v(0,\xi),D_tv(0,\xi) \big)^T$.
\begin{definition} \label{Def.Fund.Sol.}
For any $t\geq s\geq 0$, we denote by $E=E(t,s,\xi)$ and $E_V=E_V(t,s,\xi)$ the fundamental solutions of \eqref{System.with.h_1(t,xi)}
and \eqref{System.with.h_2(t,xi)}, respectively, that is, the matrix-valued functions solving the Cauchy problems
\begin{equation*}
D_tE(t,s,\xi)=A(t,\xi)E(t,s,\xi), \quad E(s,s,\xi)=I,
\end{equation*}
and
\begin{equation*}
D_tE_V(t,s,\xi)=A_V(t,\xi)E_V(t,s,\xi), \quad E_V(s,s,\xi)=I.
\end{equation*}
\end{definition}
Hence, it is easy to prove that $U(t,\xi)=E(t,0,\xi)\big( h_1(0,\xi)\fu(0,\xi),D_t\fu(0,\xi) \big)^T$ and $V(t,\xi)=E_V(t,0,\xi)\big( h_2(0,\xi)v(0,\xi),D_tv(0,\xi) \big)^T$.
\begin{remark}
By the previous considerations, after obtaining estimates for $E_V=E_V(t,s,\xi)$ it is sufficient to apply the backward transformation to the original Cauchy problem. That is, we transform back $E_V=E_V(t,s,\xi)$ to estimate the fundamental solution $E=E(t,s,\xi)$ which is related to a system of first order for the micro-energy $\big( \la(t)|\xi|\fu,D_t\fu \big)^T$, which gives the representation
\begin{equation} \label{Ell.Zone.Back.Rep.}
 E(t,s,\xi)=T(t,\xi)E_V(t,s,\xi)T^{-1}(s,\xi),
\end{equation}
where the matrix $T(t,\xi)$ is defined in the following way:
\[ \left( \begin{array}{cc}
\la(t)|\xi|\fu \\
D_t\fu
\end{array} \right)=\underbrace{\left( \begin{array}{cc}
	\frac{\la(t)|\xi|}{\de(t)h_2(t,\xi)} & 0 \\ [5pt]
	i\frac{\rho(t)\om(t)}{2\de(t)h_2(t,\xi)} & \frac{1}{\de(t)}
	\end{array} \right)}_{T(t,\xi)} \left( \begin{array}{cc}
h_2(t,\xi) v \\
D_t v
\end{array} \right) \]
with the inverse matrix
\[ T^{-1}(t,\xi)=\left( \begin{array}{cc}
\frac{\de(t)h_2(t,\xi)}{\la(t)|\xi|} & 0 \\ [5pt]
-i\frac{\rho(t)\om(t)\de(t)}{2\la(t)|\xi|} & \de(t)
\end{array} \right), \]
where the auxiliary function
\begin{equation*}
\de(t):=\exp\bigg( \frac{1}{2}\int_{0}^{t}\rho(\ta)\om(\ta)d\ta \bigg)
\end{equation*}
is related to the transformed damping term $\rho(t)\om(t)\fu_t$.
\end{remark}
\subsection{Considerations in the hyperbolic zone}
In the hyperbolic zone after $M$ steps of diagonalization procedure we can guarantee that the remainder part is uniformly integrable over the hyperbolic zone. Here we follow some ideas of \cite{Hir07}, \cite{H-W09} and \cite{Yagd.97}. First of all let us introduce the following family of symbol classes in the hyperbolic zone.
\begin{definition}
A function $f=f(t,\xi)$ belongs to the hyperbolic symbol class $S_N^l\{m_1,m_2\}$ of limited smoothness if the estimates	
\begin{equation*}
\big| D_t^k f(t,\xi) \big| \leq C_{k}\jpnxi^{m_1}\Xi(t)^{-m_2-k}
\end{equation*}
 are valid for all $(t,\xi)\in \Zhyp(N)$ and all $k=0,1,\cdots,l$ with $l\leq M$. Here $M$ is the order of the regularity of the time-dependent coefficients as well as the number of steps of the diagonalization procedure.
\end{definition}
We note that in $\Zhyp(N)$ the auxiliary symbol $\jpnxi$ can be estimated by
\begin{equation} \label{jpnxi.est.in.Hyp}
\jpnxi \approx \la(t)|\xi|.
\end{equation}
From the definition of the symbol classes we may conclude the following rules.
\begin{proposition} \label{Prop.Symbol.Hyp.}
The following statements are true:
\begin{enumerate}
\item $S_N^l\{m_1,m_2\}$ is a vector space for all non-negative integers $l$;
\item $S_N^l\{m_1,m_2\}\cdot S_N^{l'}\{m_1',m_2'\}\hookrightarrow S_N^{\tilde{l}}\{m_1+m_1',m_2+m_2'\}$
for all non-negative integers $l$ and $l'$ with $\tilde{l}=\min\{l,l'\}$;
\item $D_t^kS_N^l\{m_1,m_2\}\hookrightarrow S_N^{l-k}\{m_1,m_2+k\}$
for all non-negative integers $l$ with $k\leq l$;
\item $S_N^0\{-M,M+1\}\hookrightarrow L_{\xi}^{\iy}L_t^1\big( \Zhyp(N) \big)$ with $M$ from the assumption (A4).
\end{enumerate}
\end{proposition}
\begin{proof} We only verify the fourth property. Indeed, if $f(t,\xi)\in S_N^0\{-M,M+1\}$, then we have
\begin{align*}
\int_{t_{osc}}^{\iy}\big| f(\ta,\xi) \big|d\ta & \lesssim \int_{t_{osc}}^{\iy}\jpn_{\la(\ta),\om(\ta)}^{-M}\Xi^{-M-1}(\ta)d\ta \\
& \lesssim \int_{t_{osc}}^{\iy}|\xi|^{-M}\la^{-M}(\ta)\Xi^{-M-1}(\ta)d\ta \\
& \lesssim  |\xi|^{-M}\Th^{-M}(t_{osc}) \leq \frac{1}{N^{M}}<\iy,
\end{align*}
where we used \eqref{jpnxi.est.in.Hyp}, the assumption (A4) and the definition of $\Zhyp(N)$. \qed
\end{proof}
\begin{proposition}
Assume the conditions (A1), (A2) and (B1), (B2). Then, the following inequalities hold:
\begin{enumerate}
\item $\big| D_t^k\jpnxi \big|\lesssim \jpnxi\Xi^{-k}(t)$ for all $k=0,1,\cdots,l$ with $l\leq M$;\\
\item $\big| D_t^k\big(\rho(t)\om(t) \big) \big| \lesssim \jpnxi\Xi^{-k}(t)$ for all $k=0,1,\cdots,l$ with $l\leq M$.
\end{enumerate}
\end{proposition}
In the hyperbolic zone we have $h_2(t,\xi)=\jpnxi$. So, we introduce the micro-energy $V=\big(\jpnxi v, D_t v \big)^T$. Then, it holds
\begin{equation*}
D_tV=\left( \begin{array}{cc}
0 & \jpnxi \\ [5pt]
\jpnxi & 0
\end{array} \right)V + \left( \begin{array}{cc}
\frac{D_t\jpnxi}{\jpnxi} & 0 \\ [5pt]
-\frac{(\rho(t)\om(t))'}{2\jpnxi} & 0
\end{array} \right)V.
\end{equation*}
Let us carry out the first step of the diagonalization procedure. The eigenvalues of the first matrix are $\pm\jpnxi$. Thus, the matrix of eigenvectors $P$ and its inverse $P^{-1}$ are
\[ P = \left( \begin{array}{cc}
1 & -1 \\
1 & 1
\end{array} \right), \qquad P^{-1}=\frac{1}{2}\left( \begin{array}{cc}
1 & 1 \\
-1 & 1
\end{array} \right). \]
Defining $V^{(0)}:=P^{-1}V$, we get the transformed system
\begin{equation*}
D_tV^{(0)}=\big( \D_0(t,\xi)+\R_0(t,\xi) \big)V^{(0)},
\end{equation*}
where
\begin{equation*}
\begin{aligned}
\D_0(t,\xi) & =\left( \begin{matrix}
\jpnxi + \frac{D_t\jpnxi}{2\jpnxi}-\frac{( \rho(t)\om(t) )'}{4\jpnxi} \\ 0
\end{matrix}\right.\\
& \qquad \qquad \qquad \qquad \qquad \qquad
\left. \begin{matrix}
0 \\ -\jpnxi + \frac{D_t\jpnxi}{2\jpnxi}+\frac{( \rho(t)\om(t) )'}{4\jpnxi}
\end{matrix} \right)
\end{aligned}
\end{equation*}
and
\begin{equation*}
\R_0(t,\xi) = \left( \begin{array}{cc}
0 & -\frac{D_t\jpnxi}{2\jpnxi}+\frac{( \rho(t)\om(t) )'}{4\jpnxi} \\ [5pt]
-\frac{D_t\jpnxi}{2\jpnxi}-\frac{( \rho(t)\om(t) )'}{4\jpnxi} & 0
\end{array} \right).
\end{equation*}
Note that $\R_0(t,\xi)\in S_N^{M-1}\{0,1\}$. Now we want to carry out further steps of the diagonalization procedure. The goal is to transform the previous system such that the new matrix has diagonal structure and the new remainder belongs to a ``better'' hyperbolic symbol class. \\
The diagonalization procedure for the following lemma is essentially based on the approach used in \cite{ReiYag00} and \cite{Yagd.97} for wave equations with variable speed of propagation.
\begin{lemma}
There exists a zone constant $N>0$ such that for any $k=0,1,\cdots,M$ we can find matrices with the following properties:
\begin{itemize}
\item the matrices $N_k(t,\xi)\! \in \! S_N^{M-k}\{0,0\}$ are invertible and $N_k^{-1}(t,\xi) \!\in\! S_N^{M-k}\{0,0\}$;
\item the matrices $\D_k(t,\xi)\in S_N^{M-k}\{1,0\}$ are diagonal and
\[ \D_k(t,\xi)=\diag \big(\ta_k^+(t,\xi),\ta_k^-(t,\xi) \big) \]
with $|\ta_k^+(t,\xi)-\ta_k^-(t,\xi)|\geq C_k \jpnxi;$
\item the matrices $\R_k(t,\xi)\in S_N^{M-k}\{-k,k+1\}$ are antidiagonal;
\end{itemize}
all these matrices are defined in $\Zhyp(N)$ such that the operator identity
\begin{equation*}
\big( D_t-\D_k(t,\xi)-\R_k(t,\xi) \big)N_k(t,\xi) = N_k(t,\xi)\big( D_t-\D_{k+1}(t,\xi)-\R_{k+1}(t,\xi) \big)
\end{equation*}
is valid.
\end{lemma}
Finally, we obtain for $k=M$ the remainder $\R_M=\R_M(t,\xi)\in S_N^{0}\{-M,M+1\}$ which is uniformly integrable over the hyperbolic zone by Proposition \ref{Prop.Symbol.Hyp.}. \\
To complete the derivation of our representation we need more information on the diagonal matrices $\D_k=\D_k(t,\xi)$. An improvement of the diagonalization procedure was developed in \cite{Hir07}. The author performed more diagonalization steps in order to use structural properties of the coefficient matrices by assuming higher regularity of the entries of the matrix. Following \cite{Hir07} and \cite{H-W.Non-effective} we arrive at the following lemma.
\begin{lemma}
The difference of the diagonal entries of $\D_k(t,\xi)$ is real for all $k=0,1,\cdots,M-1$.	
\end{lemma}
Now we want to construct the fundamental solution $E_{hyp}^V=E_{hyp}^V(t,s,\xi)$ with $0 \leq s\leq t$, for the operator
\[ D_t-\D_0(t,\xi)-\R_0(t,\xi). \]
For this reason after $M$ steps of diagonalization it is sufficient to construct the fundamental solution satisfying the system
\[ D_tE_M(t,s,\xi)=\big( \D_M(t,\xi)+\R_M(t,\xi) \big)E_M(t,s,\xi), \quad E_M(s,s,\xi)=I. \]
At first we solve the diagonal system
\[ D_t\E_M(t,s,\xi)=\D_M(t,s,\xi)\E_M(t,s,\xi), \quad \E_M(s,s,\xi)=I, \quad 0\leq s\leq t. \]
Its fundamental solution is given by
\[ \E_M=\E_M(t,s,\xi)=\exp\Big( i\int_{s}^{t}\D_M(\th,\xi)d\th \Big)=\diag \Big( e^{i\int_{s}^{t}\ta_M^+(\th,\xi)d\th}, e^{i\int_{s}^{t}\ta_M^-(\th,\xi)d\th} \Big). \]
We make the \textit{ansatz} \[ E_M(t,s,\xi)=\E_M(t,s,\xi)\Q_M(t,s,\xi)\] with a uniformly bounded and invertible matrix $\Q_M=\Q_M(t,s,\xi)$. It follows that the matrix $\Q_M$ satisfies the Cauchy problem
\[ D_t\Q_M(t,s,\xi)=R_M(t,s,\xi)\Q_M(t,s,\xi), \quad \Q_M(s,s,\xi)=I \]
with the coefficient matrix
\[ R_M(t,s,\xi)=\E_M(s,t,\xi)\R_M(t,\xi)\E_M(t,s,\xi).\]
Taking account of $\R_M(t,\xi)\in S_N^0\{-M,M+1\}$ we obtain
\[ \big| R_M(t,s,\xi) \big| = \big| \R_M(t,\xi) \big|\lesssim \jpnxi^{-M}\Xi^{-M-1}(t) \lesssim |\xi|^{-M}\la^{-M}(t)\Xi^{-M-1}(t). \]
The solution $\Q_M=\Q_M(t,s,\xi)$ can be represented as Peano-Baker series
\begin{align*}
&\Q_M(t,s,\xi) \\
& \,\,\, = I+\sum_{k=1}^{\iy}i^k\int_{s}^{t}R_M(t_1,s,\xi)\int_{s}^{t_1}R_M(t_2,s,\xi)\cdots\int_{s}^{t_{k-1}}R_M(t_k,s,\xi)dt_k\cdots dt_1.
\end{align*}
Then, we obtain the following statement.
\begin{lemma} \label{Lemma Q_m}
The fundamental solution $E_{hyp}^V=E_{hyp}^V(t,s,\xi)$ is representable in the following form:
\[ E_{hyp}^V(t,s,\xi)=P\Big( \prod_{k=0}^{M-1}N_k(t,\xi) \Big)\E_M(t,s,\xi)\Q_M(t,s,\xi)\Big( \prod_{k=0}^{M-1}N_k^{-1}(s,\xi) \Big)P^{-1} \]
for all $(t,\xi), (s,\xi)\in \Zhyp(N)$, where
\begin{itemize}
\item the matrices $N_k=N_k(t,\xi)$ and $N_k^{-1}=N_k^{-1}(t,\xi)$ are uniformly bounded and invertible;
\item the matrices $\Q_M=\Q_M(t,s,\xi)$ and $\Q_M^{-1}=\Q_M^{-1}(t,s,\xi)$ are uniformly bounded and invertible.
\end{itemize}
\end{lemma}
Finally, an estimate for the fundamental solution $E_{hyp}^V=E_{hyp}^V(t,s,\xi)$ is given by the following statement.
\begin{lemma} \label{Lemma.Hyp.Before.Backward.}
Assume the conditions (A1) to (A4) and (B1) to (B3). Then, the fundamental solution $E_{hyp}^V=E_{hyp}^V(t,s,\xi)$ satisfies the estimate
\begin{equation*}
\big( |E_{hyp}^V(t,s,\xi)| \big)\lesssim \frac{\sqrt{\la(t)}}{\sqrt{\la(s)}}\left( \begin{array}{cc}
1 & 1 \\
1 & 1
\end{array} \right)
\end{equation*}
uniformly for all $(s,\xi),(t,\xi)\in \Zhyp(N)$.
\end{lemma}
After constructing the fundamental solution $E_{hyp}^V=E_{hyp}^V(t,s,\xi)$ we use the backward \quote{dissipative} transformation to the Fourier transformed original Cauchy problem \eqref{Eq.first.Fourier}. Thus, we get the following estimate for the fundamental solution $E_{hyp}=E_{hyp}(t,s,\xi)$ in the hyperbolic zone.
\begin{corollary} \label{Cor.Est.Hyp.Zone}
The fundamental solution $E_{hyp}=E_{hyp}(t,s,\xi)$ satisfies the estimate
\begin{equation*}
\big( |E_{hyp}(t,s,\xi)| \big)\lesssim \frac{\sqrt{\la(t)}}{\sqrt{\la(s)}}\exp \Big( -\frac{1}{2}\int_{s}^{t}\rho(\th)\om(\th)d\th  \Big)
\left( \begin{array}{cc}
1 & 1 \\
1 & 1
\end{array} \right)
\end{equation*}
for all $(s,\xi),(t,\xi)\in \Zhyp(N)$.
\end{corollary}
\subsection{Considerations in the oscillation subzone}
We have already chosen the hyperbolic zone as large as possible to cope with the stronger oscillating behavior of $\om=\om(t)$. For this reason we have a zone between the reduced zone $\Zred(\ve)$ and the hyperbolic zone $\Zhyp(N)$, the so-called \textit{oscillation subzone} $\Zosc(N,\ve)$. The basic approach in this zone bases on \cite{H-W09}. Essentially, in $\Zosc(N,\ve)$ we relate the fundamental solution $E_{osc}=E_{osc}(t,s,\xi)$ to the fundamental solution $E_\la=E_\la(t,s,\xi)$ to the corresponding model with $\om(t)\equiv 1$ and effective dissipation case. Note that $\Zosc(N,\ve)\subset \Zhyp^{\la}(N)$ such that we can use the known estimates for $E_\la=E_\la(t,s,\xi)$ from \cite{BuiReiBook} (see Lemma 3.5 of \cite{BuiReiBook}). This estimate reads as follows:
\begin{equation*}
\big( |E_\la(t,s,\xi)| \big)\lesssim \frac{\sqrt{\la(t)}}{\sqrt{\la(s)}}\exp \Big( -\frac{1}{2}\int_{s}^{t}\rho(\th)d\th  \Big)
\left( \begin{array}{cc}
1 & 1 \\
1 & 1
\end{array} \right).
\end{equation*}
Let us introduce the micro-energy $U(t,\xi)=\big( \la(t)|\xi|\fu,D_t\fu \big)^T$. Then, we obtain from \eqref{Eq.first.Fourier} the system of first order
\begin{equation} \label{matrix.Osc.zone}
D_tU=\underbrace{\left( \begin{array}{cc}
\frac{D_t\la(t)}{\la(t)} & \la(t)|\xi| \\
\la(t)\om^2(t)|\xi| & i\rho(t)\om(t)
\end{array} \right)}_{A(t,\xi)}U.
\end{equation}
Our aim is to construct the corresponding fundamental solution, that is, the matrix-valued solution of the system
\begin{equation*}
D_tE_{osc}(t,s,\xi)=A(t,\xi)E_{osc}(t,s,\xi), \quad E_{osc}(s,s,\xi)=I, \,\, 0\leq s\leq t.
\end{equation*}
If we set formally $\om(t)\equiv1$ and define the micro-energy to the corresponding model by $U_\la(t,\xi)=\big( \la(t)|\xi|\fu,D_t\fu \big)^T$, then it satisfies the system of first order
\begin{equation*}
D_tU_\la=\underbrace{\left( \begin{array}{cc}
\frac{D_t\la(t)}{\la(t)} & \la(t)|\xi| \\
\la(t)|\xi| & i\rho(t)
\end{array} \right)}_{A_\la(t,\xi)}U_\la.
\end{equation*}
We denote the corresponding fundamental solution as $E_\la=E_\la(t,s,\xi)$, i.e., the solution to
 \begin{equation*}
D_tE_\la(t,s,\xi)=A_\la(t,\xi)E_\la(t,s,\xi), \quad E_\la(s,s,\xi)=I, \,\, 0\leq s\leq t.
\end{equation*}
Hence, in $\Zosc(N,\ve)$ we relate $E_{osc}(t,s,\xi)$ to $E_\la(t,s,\xi)$ and use the \textit{stabilization condition} (A3).
\begin{corollary} \label{Cor.est.osc.zone}
Assume the conditions (A1) to (A3) and (B1). Then, the fundamental solution $E_{osc}=E_{osc}(t,s,\xi)$ satisfies the estimate
\begin{equation*}
\big( |E_{osc}(t,s,\xi)| \big) \lesssim \frac{\sqrt{\la(t)}}{\sqrt{\la(s)}}\exp \bigg(-\frac{1}{2}\int_{s}^{t}\rho(\th)d\th  \bigg)
\end{equation*}
uniformly for $(s,\xi), (t,\xi)\in \Zosc(N,\ve), \, \,0\leq s\leq t$.
\end{corollary}
\subsection{Considerations in the elliptic zone}
In a similar manner as in the hyperbolic zone we will try to cope with the stronger oscillating behavior of $\om=\om(t)$ by shrinking the elliptic zone. For this reason, we shall enlarge the dissipative zone in the extended phase space, since we do not propose any oscillation subzone between these two zones. \\
In the elliptic zone we can follow the standard diagonalization procedure. That is, contrary to the hyperbolic zone $\Zhyp(N)$, we will perform two diagonalization steps to derive an estimate for the fundamental solution. The considerations are based on the papers \cite{BuiReiBook} and \cite{Wirth.eff.}. \\
Now let us introduce the following family of symbol classes in $\Zell(d_0,\ve)$.
\begin{definition}
A function $f=f(t,\xi)$ belongs to the elliptic symbol class $S^l\{m_1,m_2\}$ of limited smoothness if the derivatives of $f$ satisfy the estimates
\begin{equation*}
\big| D_t^kf(t,\xi) \big|\leq C_{k}\jpnxi^{m_1}\Xi(t)^{-m_2-k}
\end{equation*}
for all $(t,\xi)\in \Zell(d_0,\ve)$ and all $k\leq l$ with $l \in \mathbb{N}_0$.
\end{definition}
Note that the auxiliary symbol $\jpnxi$ can be estimated in $\Zell(d_0,\ve)$ by
\begin{equation} \label{20}
\jpnxi \approx\frac{\rho(t)}{2}\approx\mu(t)\frac{\la(t)}{2\La(t)}.
\end{equation}
Some useful properties of the symbolic calculus are collected in the following proposition.
\begin{proposition} \label{Prop.Symbol.Ell.}  The following statements are true:
\begin{enumerate}
\item $S^l\{m_1,m_2\}$ is a vector space for all non-negative integers $l$;
\item $S^l\{m_1,m_2\}\cdot S_N^{l'}\{m_1',m_2'\}\hookrightarrow S^{\tilde{l}}\{m_1+m_1',m_2+m_2'\}$
for all non-negative integers $l$ and $l'$ with $\tilde{l}=\min\{l,l'\}$;
\item $D_t^kS^l\{m_1,m_2\}\hookrightarrow S^{l-k}\{m_1,m_2+k\}$
for all non-negative integers $l$ with $k\leq l$;
\item $S^{l-2}\{-1,2\}\hookrightarrow L_{\xi}^{\iy}L_t^1\big( \Zell(d_0,\ve) \big)$ for $l\geq2$.
\end{enumerate}
\end{proposition}
\begin{proof}
We only verify the integrability statement. Indeed, if $f=f(t,\xi)\in S^{l-2}\{-1,2\}$, then it holds
\begin{align*}
\int_{t_{diss}}^{t_{ell}} & \big| f(\ta,\xi) \big|d\ta \lesssim \int_{t_{diss}}^{t_{ell}}\frac{\Xi^{-2}(\ta)}{\jpn_{\la(\ta),\om(\ta)}}d\ta \lesssim \int_{t_{diss}}^{t_{ell}} \frac{\La(\ta)}{\mu(\ta)\la(\ta)}\frac{\la^2(\ta)F'(\La(\ta))}{F^2(\La(\ta))}d\ta \\
& \lesssim \frac{1}{|\xi|}\int_{t_{diss}}^{t_{ell}}\frac{\la(\ta)F'(\La(\ta))}{F^2( \La(\ta))}d\ta = -\frac{1}{|\xi|}\frac{1}{F\big( \La(\ta) \big)}\Big|_{t_{diss}}^{t_{ell}} \lesssim \frac{1}{|\xi|F( \La(t_{diss}))}\lesssim 1,
\end{align*}
where we used (\ref{20}), the definition of the elliptic region, $\Xi(t)=\frac{F(\La(t))}{\lambda(t)\sqrt{F'(\La(t))}}$ and $|\xi|F\big( \La(t_{diss}) \big)=d_0$, respectively. \qed
\end{proof}
In the elliptic zone we introduce the micro-energy $V=\big( \jpnxi v,D_t v \big)^T$ for all $t\geq s \geq 0$ and $(t,\xi),(s,\xi)\in \Zell(d_0,\ve)$. Then, the corresponding first order system to the Cauchy problem \eqref{after.dissipative}, with respect to the micro-energy $V$, is stated as
\begin{equation*}
D_tV=\left( \begin{array}{cc}
0 & \jpnxi \\
-\jpnxi & 0
\end{array} \right)V + \left( \begin{array}{cc}
\frac{D_t\jpnxi}{\jpnxi} & 0 \\ [5pt]
-\frac{(\rho(t)\om(t))'}{2\jpnxi} & 0
\end{array} \right)V.
\end{equation*}
Performing the diagonalization procedure we get after the second step of the diagonalization that the entries of the remainder matrix are uniformly integrable over the elliptic zone.
\paragraph{Step 1.} We denote by $P$ the matrix consisting of eigenvectors of the first matrix and its inverse $P^{-1}$. So, we have
\[ P = \left( \begin{array}{cc}
i & -i \\
1 & 1
\end{array} \right), \quad  P^{-1}=\frac{1}{2}\left( \begin{array}{cc}
-i & 1 \\
i & 1
\end{array} \right). \]
Then, defining $V^{(0)}:=\tilde{M}^{-1}V$ we get the system
\begin{equation*}
D_tV^{(0)}=\big( \D(t,\xi)+\R(t,\xi) \big)V^{(0)},
\end{equation*}
where
\begin{equation*}
\D(t,\xi)=\left( \begin{array}{cc}
-i\jpnxi & 0 \\
0 & i\jpnxi
\end{array} \right),
\end{equation*}
and
\begin{equation*}
\R(t,\xi)=\frac{1}{2} \left( \begin{array}{cc}
\frac{D_t\jpnxi}{2\jpnxi}-i\frac{( \rho(t)\om(t) )'}{4\jpnxi} & -\frac{D_t\jpnxi}{2\jpnxi}+i\frac{( \rho(t)\om(t) )'}{4\jpnxi} \\ [5pt]
-\frac{D_t\jpnxi}{2\jpnxi}-i\frac{( \rho(t)\om(t) )'}{4\jpnxi} & \frac{D_t\jpnxi}{2\jpnxi}+i\frac{( \rho(t)\om(t) )'}{4\jpnxi}
\end{array} \right).
\end{equation*}
Then, we obtain
\[ \D(t,\xi)\in S^M\{1,0\}, \quad \R(t,\xi)\in S^{M-1}\{0,1\}. \]
\paragraph{Step 2.} Let us introduce $F_0(t,\xi)=\diag\R(t,\xi)$. Now we carry out the next step(s) of the diagonalization. The difference of the diagonal entries of the matrix $\D(t,\xi)+F_0(t,\xi)$ is
\begin{equation*}
i\al(t,\xi)=2\jpnxi+\frac{\big( \rho(t)\om(t) \big)'}{2\jpnxi}\approx 2\jpnxi+\frac{\smallO{\rho^2(t)\om^2(t)}}{2\jpnxi}\approx \jpnxi
\end{equation*}
for $t\geq t_0$ with a sufficiently large $t_0=t_0(\ve)$ by using $|\big( \rho(t)\om(t) \big)'|=\smallO{\rho^2(t)\om^2(t)}$.
Now we can follow the usual diagonalization procedure. Therefore, we choose a matrix $N^{(1)}=N^{(1)}(t,\xi)$. Let
\begin{align*}
N^{(1)}(t,\xi) & = \left( \begin{array}{cc}
0 & -\frac{\R_{12}}{\al} \\
\frac{\R_{21}}{\al} & 0
\end{array} \right) \\
& \approx \left( \begin{array}{cc}
0 & i\frac{D_t\jpnxi}{4\jpnxi^2}-\frac{(\rho(t)\om(t))'}{8\jpnxi^2} \\
i\frac{D_t\jpnxi}{4\jpnxi^2}+\frac{(\rho(t)\om(t))'}{8\jpnxi^2} &0
\end{array} \right).
\end{align*}
Taking into consideration the rules of the symbolic calculus we have
\[ N^{(1)}(t,\xi)\in S^{M-1}\{-1,1\} \,\,\,\, \text{and} \,\,\,\, N_1(t,\xi)=I+N^{(1)}(t,\xi)\in S^{M-1}\{0,0\}. \]
For a sufficiently large time $t\geq t_0$ the matrix $N_1=N_1(t,\xi)$ is invertible with uniformly bounded inverse $N_1^{-1}=N_1^{-1}(t,\xi)$. Indeed, in the elliptic zone it holds
\[ \big| N_1(t,\xi)-I \big| \leq \frac{\Xi^{-1}(t)}{\jpnxi}\lesssim \frac{\La(t)}{\mu(t)\Th(t)}\ra 0 \,\,\text{for}\,\, t\ra \iy \]
due to the assumption (B5). Let
\[ B^{(1)}(t,\xi)=D_tN^{(1)}(t,\xi)-\big( \R(t,\xi)-F_0(t,\xi) \big)N^{(1)}(t,\xi)\in S^{M-2}\{-1,2\}, \]
\[ \R_1(t,\xi)=-N_1^{-1}(t,\xi)B^{(1)}(t,\xi)\in S^{M-2}\{-1,2\}. \]
Then, we have the following operator identity:
\begin{equation*}
\big( D_t-\D(t,\xi)-\R(t,\xi) \big)N_1(t,\xi)=N_1(t,\xi)\big( D_t-\D(t,\xi)-F_0(t,\xi)-\R_1(t,\xi) \big).
\end{equation*}
Hence, the previous steps of the diagonalization procedure give us the following lemma.
\begin{lemma}
Assume that $\la=\la(t)$, $\om=\om(t)$ satisfy the conditions (A1), (A2) and (A5) and $\rho=\rho(t)$ satisfies the conditions (B1), (B2) and (B5). Then, there exists a sufficiently large $t_0$ such that in $\Zell(d_0,\ve)$ the following statements hold:
\begin{itemize}
\item $N_1 \in S^{M-1}\{0,0\}$, invertible for $(t,\xi)\in\Zell(d_0,\ve)$ with $N_1^{-1} \in S^{M-1}\{0,0\};$
\item $F_0=\diag \Big( \frac{D_t\jpnxi}{2\jpnxi}-\frac{i( \rho(t)\om(t) )'}{4\jpnxi}, \frac{D_t\jpnxi}{2\jpnxi}+\frac{i( \rho(t)\om(t) )'}{4\jpnxi} \Big)\in S^{M-1}\{0,1\};$
\item $\R_1 \in S^{M-2}\{-1,2\}$ with $M\geq 2$.
\end{itemize}
Moreover, the operator identity
\[ \big( D_t-\D(t,\xi)-\R(t,\xi) \big)N_1(t,\xi)=N_1(t,\xi)\big( D_t-\D(t,\xi)-F_0(t,\xi)-\R_1(t,\xi) \big) \]	
holds for all $(t,\xi)\in\Zell(d_0,\ve)$.
\end{lemma}
\paragraph{Step 3. Construction of the fundamental solution.} In order to solve the transformed system and construct its fundamental solution we can not follow the considerations from the theory of the hyperbolic zone, since the main diagonal entries are purely imaginary.
\begin{lemma} \label{Lem.Est.Ell.Zone}
Assume the conditions (A1), (A2), (A5) and (B1), (B2), (B5). Then, the fundamental solution $E_{ell}^{V}=E_{ell}^{V}(t,s,\xi)$ to the transformed operator
\[ D_t-\D(t,\xi)-F_0(t,\xi)-\R_1(t,\xi) \]
can be estimated by
\begin{equation*}
\big( |E_{ell}^{V}(t,s,\xi)| \big) \lesssim \frac{\jpnxi}{\jpn_{\la(s),\om(s)}} \exp \bigg( \int_{s}^{t}\jpn_{\la(\ta),\om(\ta)}d\ta \bigg)
\left( \begin{array}{cc}
1 & 1 \\
1 & 1
\end{array} \right),
\end{equation*}
with $(t,\xi),(s,\xi)\in \Zell(d_0,\ve)\cap \{t\geq t_0(\ve)\}, \,0\leq s\leq t$.
\end{lemma}
In order to prove Lemma \ref{Lem.Est.Ell.Zone} one can proceed in the same manner as in the papers \cite{BuiReiBook} and \cite{Wirth.eff.}.
\paragraph{Step 4. Transforming back to the original Cauchy problem.} Now we want to obtain an estimate for the energy of the solution to our original Cauchy problem. For this reason we need to transform back to get an estimate of the fundamental solution $E_{ell}=E_{ell}(t,s,\xi)$ which is related to a system of first order for the micro-energy $\big( \la(t)|\xi|\fu,D_t\fu \big)^T.$
\begin{lemma} \label{Lem.Ell.Aux.Est.}
Under the assumptions (B1) to (B3) the following holds:
\begin{itemize}
\item [1.] In the elliptic zone it holds $\jpnxi-\dfrac{\rho(t)\om(t)}{2} \leq -\dfrac{\la^2(t)\om(t)|\xi|^2}{\rho(t)}$,
\item [2.] $\dfrac{\de(s)}{\de(t)}\exp \Big( \displaystyle \int_{s}^{t}\jpn_{\la(\ta),\om(\ta)}d\ta \Big)\leq \exp
\Big( -|\xi|^2\displaystyle \int_{s}^{t}\dfrac{\la^2(\ta)\om(\ta)}{\rho(\ta)}d\ta \Big),$
\end{itemize}
where $\de=\de(t)=\exp\big( \frac{1}{2}\int_{0}^{t}\rho(\ta)\om(\ta)d\ta \big)$.	
\end{lemma}
\begin{proof} By using the property
\[ \sqrt{x+y}\leq \sqrt{x}+\frac{y}{2\sqrt{x}} \]
for any $x\geq 0$ and $y\geq -x$, the first statement is equivalent to the following inequality:
\[ \sqrt{ \frac{\rho^2(t)\om^2(t)}{4}-\la^2(t)\om^2(t)|\xi|^2 }-\frac{\rho(t)\om(t)}{2}\leq -\frac{\la^2(t)\om(t)|\xi|^2}{\rho(t)}. \]
After integration the second statement follows directly from the first one together with the definition of $\de=\de(t)$. \qed
\end{proof}
From Lemma \ref{Lem.Est.Ell.Zone} we get for $(t,\xi), (s,\xi)\in \Zell(d_0,\ve)$ the estimate
\[ \big( |E_{ell}^V(t,s,\xi)| \big) \lesssim \frac{\rho(t)\om(t)}{\rho(s)\om(s)}\exp \Big( \int_{s}^{t}\jpn_{\la(\ta),\om(\ta)}d\ta \Big)\left( \begin{array}{cc}
1 & 1 \\
1 & 1
\end{array} \right). \]
This yields in combination with \eqref{Ell.Zone.Back.Rep.} the estimate
\begin{align} \nonumber
\big(& |E_{ell}(t,s,\xi)| \big) \\ \nonumber
& \lesssim \left( \begin{array}{cc}
\la(t)|\xi| & 0 \\ [5pt]
\rho(t) & \rho(t)
\end{array} \right)\exp \bigg( \int_{s}^{t}\Big( \jpn_{\la(\ta),\om(\ta)}-\frac{\rho(\ta)\om(\ta)}{2} \Big)d\ta \bigg)
\left( \begin{array}{cc}
1 & 1 \\
1 & 1
\end{array} \right)\left( \begin{array}{cc}
\frac{1}{\la(s)|\xi|} & 0 \\ [5pt]
\frac{1}{\la(s)|\xi|} & \frac{1}{\rho(s)}
\end{array} \right) \\ \label{Ell.Zone.Before.Ref.Est.}
& \lesssim \exp \Big( -|\xi|^2\int_{s}^{t}\frac{\la^2(\ta)\om(\ta)}{\rho(\ta)}d\ta \Big) \left( \begin{array}{cc}
\frac{\la(t)}{\la(s)} & \frac{\la(t)|\xi|}{\rho(s)} \\ [5pt]
\frac{\rho(t)}{\la(s)|\xi|} & \frac{\rho(t)}{\rho(s)}
\end{array} \right).
\end{align}
\begin{remark}
Taking into account the estimates \eqref{Ell.Zone.Before.Ref.Est.}, we see that the estimates for the first row are reasonable. However, the estimates for the second row seem to be not reasonable. Because, the estimate for $|E_{ell}^{(22)}|$ is only reasonable for decreasing coefficients $\rho=\rho(t)$ and the estimate for $|E_{ell}^{(21)}|$ is not optimal since the upper bound for $|E_{ell}^{(21)}|$ is not bounded in the elliptic zone. This \quote{contradicts somehow} the damping effect in our model. For this reason we derive a refined estimate which  we present in the next step.
\end{remark}
\paragraph{Step 5. A refined estimate for the fundamental solution in the elliptic zone.}
\begin{corollary} \label{Cor.Est.Ell.Zone}
The fundamental solution $E_{ell}=E_{ell}(t,s,\xi)$ satisfies the following estimate:
\begin{equation*}
\big( |E_{ell}(t,s,\xi)| \big) \lesssim \exp \Big( -|\xi|^2\int_{s}^{t}\frac{\la^2(\ta)\om(\ta)}{\rho(\ta)}d\ta \Big) \left( \begin{array}{cc}
\frac{\la(t)}{\la(s)} & \frac{\la(t)|\xi|}{\rho(s)} \\ [5pt]
\frac{\la^2(t)|\xi|}{\la(s)\rho(t)} & \frac{\la^2(t)|\xi|^2}{\rho(s)\rho(t)}
\end{array} \right)+\frac{\de^2(s)}{\de^2(t)}\left( \begin{array}{cc}
0 & 0 \\
0 & 1
\end{array} \right)
\end{equation*}
for all $t\geq s$ and $(t,\xi), (s,\xi)\in\Zell(d_0,\ve)$.
\end{corollary}
\begin{proof} Let us assume that $\Phi_k=\Phi_k(t,s,\xi)$, $k=1,2,$ are solutions to the equation
\[ \Phi_{tt}+\la^2(t)\om^2(t)|\xi|^2\Phi+\rho(t)\om(t)\Phi_t=0 \]
with initial values $\Phi_k(s,s,\xi)=\de_{1k}$ and $\pa_t\Phi_k(s,s,\xi)=\de_{2k}$. Then, we have
\[ \left( \begin{array}{cc}
    \la(t)|\xi|v(t,\xi) \\
	D_tv(t,\xi)
	\end{array} \right)=\left( \begin{array}{cc}
	\dfrac{\la(t)}{\la(s)}\Phi_1(t,s,\xi) & i\la(t)|\xi|\Phi_2(t,s,\xi) \\ [5pt]
	\dfrac{D_t\Phi_1(t,s,\xi)}{\la(s)|\xi|} & iD_t\Phi_2(t,s,\xi)
	\end{array} \right) \left( \begin{array}{cc}
	\la(s)|\xi|v(s,\xi) \\
	D_tv(s,\xi)
	\end{array} \right). \]
Our basic idea is to relate the entries of the above given estimates to the multipliers $\Phi_k=\Phi_k(t,s,\xi)$ and use Duhamel's formula to improve the estimates for the second row using estimates from the first one (see \cite{Wirth.eff.}). \qed
\end{proof}
\begin{remark}
We are able to derive a refined estimate for the fundamental solution, because in the proof of Corollary \ref{Cor.Est.Ell.Zone} we use only estimates for $E_{ell}^{(11)}$ and $E_{ell}^{(12)}$ and both estimates seem to be optimal with our analytical tools.
\end{remark}
\begin{remark}
If we choose a fixed $s$, then the second summand in Corollary \ref{Cor.Est.Ell.Zone} is dominated by the first one. Indeed, if we set $s=t_{ell}$, then by using $\la(t_{ell})|\xi|\approx \rho(t_{ell})$ we get the following estimate:
\begin{equation*}
\big( |E_{ell}(t,s,\xi)| \big)\lesssim \exp \Big( -|\xi|^2\int_{t_{ell}}^{t}\frac{\la^2(\ta)\om(\ta)}{\rho(\ta)}d\ta \Big)
\left( \begin{array}{cc}
\frac{\la(t)}{\la(t_{ell})} & \frac{\la(t)}{\la(t_{ell})} \\ [5pt]
\frac{\la^2(t)|\xi|}{\la(t_{ell})\rho(t)} & \frac{\la^2(t)|\xi|}{\la(t_{ell})\rho(t)}
\end{array} \right).
\end{equation*}
\end{remark}
\subsection{Considerations in the dissipative zone}
In the dissipative zone we define the micro-energy $U=U(t,\xi)$ by
\[ U=\big( \ga(t)\fu,D_t\fu \big)^T, \quad \ga(t):=\frac{\la(t)}{F( \La(t))}. \]
This seems to be reasonable because we will later need to estimate $\la(t)|\xi|\fu$ and it holds $\la(t)|\xi|\lesssim \dfrac{\la(t)}{F( \La(t))}$ due to the definition of the dissipative zone. Then, the Fourier transformed Cauchy problem \eqref{Eq.first.Fourier} leads to the system of first order
\begin{equation} \label{system.diss.zone}
D_tU=\underbrace{\left( \begin{array}{cc}
\frac{D_t\ga(t)}{\ga(t)} & \ga(t) \\ [5pt]
\frac{\la^2(t)\om^2(t)|\xi|^2}{\ga(t)} & i\rho(t)\om(t)
\end{array} \right)}_{A(t,\xi)}U.
\end{equation}
We are interested in the fundamental solution \[ E_{diss}=E_{diss}(t,s,\xi)=\left( \begin{matrix}
E_{diss}^{(11)} & E_{diss}^{(12)} \\
E_{diss}^{(21)} & E_{diss}^{(22)}
\end{matrix} \right) \]
to the system \eqref{system.diss.zone}, that is, the solution of
\[ D_tE_{diss}(t,s,\xi)=A(t,\xi)E_{diss}(t,s,\xi), \,\,\, E_{diss}(s,s,\xi)=I, \]
for all $0\leq s \leq t$ and $(t,\xi), (s,\xi) \in \Zdiss(d_0)$. Thus, the solution $U=U(t,\xi)$ is represented as
\[ U(t,\xi)=E_{diss}(t,s,\xi)U(s,\xi). \]
We will use the auxiliary function
\[ \de(t)=\exp\Big( \frac{1}{2}\int_{0}^{t}\rho(\ta)\om(\ta)d\ta \Big) \]
which is related to the entry $i\rho(t)\om(t)$ of the coefficient matrix. \\
The entries $E_{diss}^{(kl)}(t,s,\xi)$, $k,l=1,2,$ of the fundamental solution $E_{diss}(t,s,\xi)$ satisfy the following system of Volterra integral equations for $k=1,2$:
\begin{align*}
D_tE_{diss}^{(1l)}(t,s,\xi) & = \frac{D_t\ga(t)}{\ga(t)}E_{diss}^{(1l)}(t,s,\xi)+\ga(t)E_{diss}^{(2l)}(t,s,\xi), \\
D_tE_{diss}^{(2l)}(t,s,\xi) & = \frac{\la^2(t)\om^2(t)|\xi|^2}{\ga(t)}E_{diss}^{(1l)}(t,s,\xi)+i\rho(t)\om(t)E_{diss}^{(2l)}(t,s,\xi)
\end{align*}
together with their initial conditions
\[ \left( \begin{array}{cc}
E_{diss}^{(11)}(s,s,\xi) & E_{diss}^{(12)}(s,s,\xi) \\
E_{diss}^{(21)}(s,s,\xi) & E_{diss}^{(22)}(s,s,\xi)
\end{array} \right)=\left( \begin{array}{cc}
1 & 0 \\
0 & 1
\end{array} \right). \]
Then, by direct calculations we get
\begin{align*}
E_{diss}^{(11)}(t,s,\xi) & = \frac{\ga(t)}{\ga(s)}+i\ga(t)\int_{s}^{t}E_{diss}^{(21)}(\ta,s,\xi)d\ta, \\
E_{diss}^{(21)}(t,s,\xi) & = \frac{i|\xi|^2}{\de^2(t)}\int_{s}^{t}\frac{\la^2(\ta)\om^2(\ta)}{\ga(\ta)}\de^2(\ta)E_{diss}^{(11)}(\ta,s,\xi)d\ta, \\
E_{diss}^{(12)}(t,s,\xi) & = i\ga(t)\int_{s}^{t}E_{diss}^{(22)}(\ta,s,\xi)d\ta, \\
E_{diss}^{(22)}(t,s,\xi) & = \frac{\de^2(s)}{\de^2(t)}+\frac{i|\xi|^2}{\de^2(t)}\int_{s}^{t}\frac{\la^2(\ta)\om^2(\ta)}{\ga(\ta)}\de^2(\ta)E_{diss}^{(12)}(\ta,s,\xi)d\ta.
\end{align*}
The next lemma is important for deriving suitable estimates for the entries $E_{diss}^{(kl)}(t,s,\xi)$, $k,l=1,2$.
\begin{lemma} \label{Diss.Zone.Lemma}
The assumption (B3) implies $\frac{\la(t)}{\de^2(t)}\in L^1\mathbb{(R_+)}$ with
\begin{equation*}
\int_{t}^{\iy}\frac{\la(\ta)}{\de^2(\ta)}d\ta\lesssim \frac{\La(t)}{\de^2(t)}.
\end{equation*}
Moreover, $\frac{\La(t)}{\de^2(t)}$ is monotonously decreasing for large $t$.
\end{lemma}
\begin{proof}
From $\mu(t)\ra\iy$ as $t\ra\iy$ it follows $\mu(t)\geq\frac{1+\ve}{\om(t)}$. Then, we may conclude
\[ \de^2(t)=\exp \Big( \int_{0}^{t}\rho(\ta)\om(\ta)d\ta \Big) \gtrsim \exp \Big( (1+\ve) \int_{0}^{t}\frac{\la(\ta)}{\La(\ta)}d\ta \Big)=\La^{1+\ve}(t),\]
which implies the integrability of $\frac{\la(t)}{\de^2(t)}$. Furthermore, for large $t$ we have
\begin{align*}
\frac{1}{C_\varepsilon}&\int_{t}^{\iy}\frac{\la(\ta)}{\de^2(\ta)}d\ta \leq \int_{t}^{\iy}\frac{\ve}{\om(\ta)}\frac{\la(\ta)}{\de^2(\ta)}d\ta \leq \int_{t}^{\iy}\Big( \mu(\ta)-\frac{1}{\om(\ta)}\Big)\frac{\la(\ta)}{\de^2(\ta)}d\ta \\
& = \int_{t}^{\iy} \frac{\mu(\ta)\la(\ta)\om(\ta)-\la(\ta)}{\om(\ta)\de^2(\ta)} d\ta \lesssim \int_{t}^{\iy} \frac{\mu(\ta)\la(\ta)\om(\ta)-\la(\ta)}{\de^2(\ta)} d\ta = \frac{\La(t)}{\de^2(t)}.
\end{align*}
Moreover, we have
\[ \frac{d}{dt}\frac{\La(t)}{\de^2(t)}=\frac{\la(t)-\rho(t)\om(t)\La(t)}{\de^2(t)}=\frac{\la(t)\big( 1-\mu(t)\om(t) \big)}{\de^2(t)} \]
and $\mu(t)\geq \frac{1+\ve}{\om(t)}$ for large $t$, which implies that $\dfrac{\La(t)}{\de^2(t)}$ is decreasing for large $t$. This completes the proof. \qed
\end{proof}
In order to estimate the modulus $\big| E_{diss}^{(kl)}(t,s,\xi) \big|$, $k,l=1,2,$ of the entries $E_{diss}^{(kl)}(t,s,\xi)$ we will use the assumption (B6).
\begin{corollary} \label{Cor.Est.Diss.Zone}
Assume the conditions (A1) for $\la(t)$, (A2) for $\om(t)$, (B3) and (B6) for $\rho(t)$. Then, we have the following estimate in the dissipative zone:
\begin{equation*}
\big( |E_{diss}(t,s,\xi)| \big) \lesssim \frac{\la(t)}{F(\La(t))} \left( \begin{array}{cc}
\frac{F(\La(s))}{\la(s)} & \frac{\La(s)}{\la(s)} \\ [5pt]
\frac{F(\La(s))}{\la(s)} & \frac{\La(s)}{\la(s)}
\end{array} \right),
\end{equation*}
with $(s,\xi),(t,\xi)\in\Zdiss(d_0)$ and $0\leq s\leq t$.
\end{corollary}
\begin{proof}
First let us consider the first column. Plugging the representation for $E_{diss}^{(21)}(t,s,\xi)$ into the integral equation for $E_{diss}^{(11)}(t,s,\xi)$ gives
\[ \frac{\ga(s)}{\ga(t)}E_{diss}^{(11)}(t,s,\xi)=1-|\xi|^2\int_{s}^{t}\int_{s}^{\ta}\la^2(\th)\om^2(\th)\frac{\de^2(\th)}{\de^2(\ta)}\frac{\ga(s)}{\ga(\th)}
E_{diss}^{(11)}(\th,s,\xi)d\th d\ta. \]
By setting $y(t,s,\xi):=\dfrac{\ga(s)}{\ga(t)}E_{diss}^{(11)}(t,s,\xi)$ and applying partial integration we obtain
\begin{align*}
y(t,s,\xi) & = 1+|\xi|^2\int_{s}^{t}\la^2(\th)\om^2(\th)\Big( \frac{1}{\rho(\ta)\om(\ta)}\frac{\de^2(\th)}{\de^2(\ta)}\Big|_{\th}^{t} \\
& \qquad \qquad \qquad \qquad \qquad \qquad + \int_{\th}^{t}\underbrace{\frac{\big( \rho(\ta)\om(\ta) \big)'}{\big( \rho(\ta)\om(\ta) \big)^2}}_{=\smallO{1} \, \mbox{for}\, \tau \to \infty} \frac{\de^2(\th)}{\de^2(\ta)} d\ta \Big) y(\th,s,\xi)d\th \\
& \approx 1+|\xi|^2\int_{s}^{t}\la^2(\th)\om^2(\th)\Big( \frac{1}{\rho(t)\om(t)}\frac{\de^2(\th)}{\de^2(t)}-\frac{1}{\rho(\th)\om(\th)} \Big)y(\th,s,\xi)d\th.
\end{align*}
Then, we have
\[ |y(t,s,\xi)|\lesssim 1+|\xi|^2\int_{s}^{t}\Big( \frac{\la^2(\th)\om^2(\th)}{\rho(t)\om(t)}\frac{\de^2(\th)}{\de^2(t)}+\frac{\la^2(\th)\om(\th)}{\rho(\th)} \Big)|y(\th,s,\xi)|d\th. \]
Applying Gronwall's inequality and partial integration, respectively, we get
\begin{align*}
|y(t,s,\xi)| & \lesssim \exp \Big( \frac{\la^2(t)|\xi|^2}{\rho(t)}\frac{1}{\de^2(t)}\int_{s}^{t}\de^2(\th)d\th+|\xi|^2\int_{s}^{t}\frac{\la^2(\th)}{\rho(\th)}d\th \Big) \\
& = \exp \bigg( \frac{\la^2(t)|\xi|^2}{\rho(t)}\frac{1}{\de^2(t)} \Big( \frac{\de^2(\th)}{\rho(\th)\om(\th)}\Big|_{s}^{t} \\
& \qquad \qquad +\int_{s}^{t}\frac{\big( \rho(\th)\om(\th) \big)'}{\big( \rho(\th)\om(\th) \big)^2}\de^2(\th) d\th \Big) + |\xi|^2\int_{s}^{t}\frac{\la^2(\th)}{\rho(\th)}d\th \bigg) \\
& \lesssim \exp \Big( C\frac{\la^2(t)|\xi|^2}{\rho^2(t)} + C\frac{d_0^2}{F^2\big( \La(t) \big)}\int_{s}^{t}\frac{\la^2(\th)}{\rho(\th)}d\th \Big) \lesssim 1.
\end{align*}
Here we have used $\frac{(\rho(\th)\om(\th))'}{(\rho(\th)\om(\th))^2}=\smallO{1}$ from the assumption (B5), $\la(t)|\xi|\lesssim \rho(t)$ from the definition of the elliptic region, $|\xi|\leq\frac{d_0}{F(\La(t))}$ from the definition of the dissipative zone and the assumption (B6), respectively. Hence, we may conclude
\[ \big| E_{diss}^{(11)}(t,s,\xi) \big| \lesssim \frac{\ga(t)}{\ga(s)}=\frac{\la(t)}{F( \La(t))}\frac{F(\La(s))}{\la(s)}. \]
Now we consider $E_{diss}^{(21)}(t,s,\xi)$. By using the estimate for $\big| E_{diss}^{(11)}(t,s,\xi) \big|$ we obtain
\begin{align*}
\big| E_{diss}^{(21)}(t,s,\xi) \big| & \lesssim |\xi|^2\int_{s}^{t}\frac{\la^2(\ta)}{\ga(\ta)}\frac{\de^2(\ta)}{\de^2(t)}\big| E_{diss}^{(11)}(\ta,s,\xi) \big|d\ta\lesssim \frac{1}{\ga(s)}\frac{\la^2(t)|\xi|^2}{\de^2(t)}\int_{s}^{t}\de^2(\ta)d\ta \\
& = \frac{1}{\ga(s)}\frac{\la^2(t)|\xi|^2}{\de^2(t)}\bigg( \frac{1}{\rho(\ta)\om(\ta)}\de^2(\ta)\Big|_{s}^{t}+\int_{s}^{t}\frac{\big( \rho(\ta)\om(\ta) \big)'}{\big( \rho(\ta)\om(\ta) \big)^2}\de^2(\ta)d\ta \bigg) \\
& \lesssim \frac{1}{\ga(s)}\frac{\la^2(t)|\xi|^2}{\rho(t)\om(t)} \lesssim \frac{1}{\ga(s)}\la(t)|\xi| \lesssim \frac{F\big( \La(s) \big)}{\la(s)}\frac{\la(t)}{F\big( \La(t) \big)}.
\end{align*}
Here we used $\frac{(\rho(\th)\om(\th))'}{(\rho(\th)\om(\th))^2}=\smallO{1}$ from the assumption (B5), $\la(t)|\xi|\lesssim \rho(t)$ from the definition of the elliptic region and $|\xi|\leq\frac{d_0}{F(\La(t))}$ from the definition of the dissipative zone, respectively. \\
Next we consider the entries of the second column. Plugging the representation for $E_{diss}^{(22)}(t,s,\xi)$ into the integral equation for $E_{diss}^{(12)}(t,s,\xi)$ gives
\begin{align*}
& E_{diss}^{(12)}(t,s,\xi) \\
& = i\ga(t)\de^2(s)\int_{s}^{t}\frac{d\ta}{\de^2(\ta)}-
|\xi|^2\ga(t)\int_{s}^{t}\int_{s}^{\ta}\la^2(\th)\om^2(\th)\frac{\de^2(\th)}{\de^2(\ta)}\frac{1}{\ga(\th)}E_{diss}^{(12)}(\th,s,\xi)d\th d\ta.
\end{align*}
By setting $y(t,s,\xi):=\dfrac{1}{\ga(t)}E_{diss}^{(12)}(t,s,\xi)$ and proceeding in the same manner as with $E_{diss}^{(11)}(t,s,\xi)$, after integration by parts we obtain
\begin{align*}
|y(t,s,\xi)| & \lesssim \de^2(s)\int_{s}^{t}\frac{\la(\ta)}{\de^2(\ta)}\frac{1}{\la(\ta)}d\ta \\
& \qquad + |\xi|^2\int_{s}^{t}\Big( \frac{\la^2(\th)\om^2(\th)}{\rho(t)\om(t)}\frac{\de^2(\th)}{\de^2(t)}+\frac{\la^2(\th)\om^2(\th)}{\rho(\th)\om(\th)} \Big)|y(\th,s,\xi)|d\th.
\end{align*}
Now we can use again Gronwall's inequality, since the first integral can be estimated by $\dfrac{\La(s)}{\de^2(s)}\dfrac{1}{\la(s)}$ due to Lemma \ref{Diss.Zone.Lemma}. Then, we get
\begin{align*}
|y(t,s,\xi)| \lesssim \frac{\La(s)}{\la(s)}.
\end{align*}
Thus, we obtain
\[ \big| E_{diss}^{(12)}(t,s,\xi) \big|\lesssim \ga(t)\frac{\La(s)}{\la(s)}=\frac{\la(t)}{F\big( \La(t) \big)}\frac{\La(s)}{\la(s)}. \]
Finally, let us consider $E_{diss}^{(22)}(t,s,\xi)$ by using the estimate for $\big| E_{diss}^{(12)}(t,s,\xi) \big|$. In the same way as in the estimate for $\big| E_{diss}^{(21)}(t,s,\xi) \big|$ we obtain
\begin{align*}
\big| E_{diss}^{(22)}&(t,s,\xi) \big| \lesssim \frac{\de^2(s)}{\de^2(t)}+|\xi|^2\int_{s}^{t}\la^2(\ta)\frac{\de^2(\ta)}{\de^2(t)}\frac{1}{\ga(\ta)}\big| E_{diss}^{(12)}(\ta,s,\xi) \big|d\ta \\
& \lesssim \frac{\de^2(s)}{\de^2(t)}+\frac{\la^2(t)|\xi|^2}{\de^2(t)}\frac{\La(s)}{\la(s)}\int_{s}^{t}\de^2(\ta)d\ta \lesssim \frac{\de^2(s)}{\de^2(t)}+\frac{\La(s)}{\la(s)}\la(t)|\xi|.
\end{align*}
We rewrite the last inequality as
\begin{align*}
\frac{\La(t)}{\la(t)}\big| E_{diss}^{(22)}(t,s,\xi) \big| \lesssim \frac{\La(s)}{\la(s)}+\frac{\La(s)}{\la(s)}\La(t)|\xi|\lesssim \frac{\La(s)}{\la(s)}+\frac{\La(s)}{\la(s)}\frac{\La(t)}{F\big( \La(t) \big)},
\end{align*}
where we used that $\dfrac{\La(t)}{\de^2(t)}$ is decreasing for large $t$ due to Lemma \ref{Diss.Zone.Lemma}. Thus, we get
\begin{align*}
\big| E_{diss}^{(22)}(t,s,\xi) \big| \lesssim \frac{\La(s)}{\la(s)}\frac{\la(t)}{\La(t)}+\frac{\La(s)}{\la(s)}\frac{\La(t)}{F\big( \La(t) \big)}\frac{\la(t)}{\La(t)} \lesssim \frac{\la(t)}{F\big( \La(t) \big)}\frac{\La(s)}{\la(s)},
\end{align*}
where we used $\La(t)\geq F\big( \La(t) \big)$. This completes the proof. \qed
\end{proof}
\subsection{Considerations in the reduced zone}
In the reduced zone we introduce the micro-energy $V=V(t,\xi)$ by
\[ V=\Big( \ve\frac{\rho(t)\om(t)}{2}v, D_t v \Big)^T. \]
Then, by \eqref{after.dissipative} the function $V$ satisfies the following system:
\begin{equation} \label{Red.zone.system.A.V}
D_tV=\underbrace{\left( \begin{array}{cc}
\frac{D_t(\rho(t)\om(t))}{\rho(t)\om(t)} & \ve\frac{\rho(t)\om(t)}{2} \\ [5pt]
\frac{\la^2(t)\om^2(t)|\xi|^2-\frac{1}{4}(\rho(t)\om(t))^2-\frac{1}{2}(\rho(t)\om(t))'}{\ve\frac{\rho(t)\om(t)}{2}} & 0
\end{array} \right)}_{A_V(t,\xi)}V.
\end{equation}
We want to estimate the fundamental solution $E_{red}^V=E_{red}^V(t,s,\xi)$ to \eqref{Red.zone.system.A.V}, that is, the solution to
\[ D_tE_{red}^V(t,s,\xi)=A_V(t,\xi)E_{red}^V(t,s,\xi), \quad E_{red}^V(s,s,\xi)=I. \]
Due to $\big| \big( \rho(t)\om(t) \big)' \big|=\smallO{(\rho(t)\om(t))^2}$ for sufficiently large $t\geq t_0$, it holds
\[ \frac{\big| D_t\big( \rho(t)\om(t) \big) \big|}{\rho(t)\om(t)}\lesssim \ve\rho(t)\om(t). \]
Moreover, we have the estimate
\[ \jpnxi\lesssim \ve\frac{\rho(t)\om(t)}{2}. \]
Hence, we obtain the following estimate:
\begin{align*}
\frac{\big| \la^2(t)\om^2(t)|\xi|^2-\frac{1}{4}(\rho(t)\om(t))^2-\frac{1}{2}(\rho(t)\om(t))' \big|}{\ve\frac{\rho(t)\om(t)}{2}} \lesssim \ve\rho(t)\om(t),
\end{align*}
where we used $\big|\big(\rho(t)\om(t)\big)'\big|=\smallO{\left(\rho(t)\om(t)\right)^2}$. Finally, the norm of the coefficient matrix of \eqref{Red.zone.system.A.V} can be estimated by $\ve\rho(t)\om(t)$ for sufficiently large $t$.
\begin{remark}
From the backward transformation we may conclude that the fundamental solution $E_{red}=E_{red}(t,s,\xi)$ can be estimated as follows:
\begin{align*}
\big( |E_{red}(t,s,\xi)| \big) & \lesssim \exp \Big( -\frac{1}{2}\int_{s}^{t}\rho(\ta)\om(\ta)d\ta \Big)\big( |E_{red}^V(t,s,\xi)| \big).
\end{align*}
\end{remark}
\begin{corollary} \label{Cor.Red.zone}
Under the assumptions (B1), (B2) and (B5) the fundamental solution $E_{red}=E_{red}(t,s,\xi)$ satisfies the following estimate in the reduced zone:
\[ \big( |E_{red}(t,s,\xi)| \big) \lesssim \exp \bigg( \Big( \ve-\frac{1}{2} \Big) \int_{s}^{t}\rho(\ta)\om(\ta)d\ta \bigg)\left( \begin{array}{cc}
    1 & 1 \\
    1 & 1
\end{array} \right) \]
for $t\geq s \geq t_0$ with a sufficiently large $t_0=t_0(\ve)$ and $(t,\xi), (s,\xi)\in\Zred(\ve)$.
\end{corollary}
\section{Gluing procedure}
In the previous sections we have derived estimates for fundamental fundamental solutions in different zones. Now we have to glue these estimates from Corollaries \ref{Cor.Est.Hyp.Zone}, \ref{Cor.est.osc.zone}, \ref{Cor.Est.Ell.Zone}, \ref{Cor.Est.Diss.Zone} and \ref{Cor.Red.zone}. \\
Taking into account of the estimates in $\Zhyp(N)$, $\Zosc(N,\ve)$ and $\Zred(\ve)$ we can uniformly estimate $\big( |E_{hyp}(t,s,\xi)| \big)$ and $\big( |E_{osc}(t,s,\xi)| \big)$ by the upper bound from the estimate for $\big( |E_{red}(t,s,\xi)| \big)$. Therefore, we can glue $\Zred(\ve)$ to the hyperbolic region and we define new regions by
\begin{align*}
\Pi_{hyp}(N,\ve) & = \Zred(\ve)\cup \Zosc(N,\ve) \cup \Zhyp(N), \\
\Pi_{ell}(d_0,\ve) & = \Zell(d_0,\ve)\cup Z_{diss}(d_0).
\end{align*}
We denote by $t_{diss}(|\xi|)=:t_{diss}$ the separating line between $\Zell(d_0,\ve)$ and $\Zdiss(d_0)$ and by $t(|\xi|)=:t_{|\xi|}$ the separating curve between $\Pi_{ell}(d_0,\ve)$ and $\Pi_{hyp}(N,\ve)$. Due to the definitions of the zones these separating lines really exist and can be described by functions due to the definitions of the zones and the monotonicity of these functions. Indeed, denoting $\eta(t):=\frac{\mu(t)}{2\La(t)}$, this curve is given by
\[ \eta^2(t_{|\xi|})-|\xi|^2 = \ve^2\eta^2(t_{|\xi|}), \quad \text{i.e.,} \quad t_{|\xi|}=\eta^{-1}\Big( \frac{|\xi|}{\sqrt{1-\ve^2}} \Big). \]
\begin{definition}
We denote by $B_\la=B_\la(s,t)$, $0\leq s \leq t$, the primitive of $\dfrac{\la^2(t)}{\rho(t)}$ which vanishes at $t=s$. So, it is defined by
\[ B_\la(s,t):=\int_{s}^{t}\frac{\la^2(\ta)}{\rho(\ta)}d\ta=B_\la(0,t)-B_\la(0,s). \]
\end{definition}
In order to obtain energy estimates, first we establish some auxiliary estimates.
\begin{lemma} \label{Gluing.Auxiliary.Lemma}
Under the assumptions (A1), (B1), (B2) and (B6) the following estimates hold:
\begin{itemize}
\item [1.] Supposing $|\xi|F( \La(t_{diss}))=d_0$ it holds
\[ \exp \Big( -C|\xi|^2\int_{0}^{t_{diss}}\frac{\la^2(\ta)}{\rho(\ta)}d\ta \Big)\approx 1. \]
\item [2.] Supposing $|\xi|=\eta(t_{|\xi|})\sqrt{1-\ve^2}$ it holds
\[ \big|d_{|\xi|}t_{|\xi|}\big| \gtrsim \frac{\mu(t_{|\xi|})}{|\xi|\rho(t_{|\xi|})}. \]
\end{itemize}
\end{lemma}
\textbf{Case 1: the function $\eta=\eta(t)$ is monotonously decreasing} \smallskip

Now we distinguish between two cases related to the setting of the zones in the extended phase space. \smallskip

\textbf{Large frequencies} \smallskip

In this case the large frequencies are in $\Zhyp(N)$ only. Then, it follows
\[ \big( |E(t,0,\xi)| \big)\lesssim \Big( \frac{1}{\de(t)} \Big)^{1-2\ve}
\left( \begin{array}{cc}
1 & 1 \\
1 & 1 \end{array} \right), \]
where $\de(t)=\exp\big( \frac{1}{2}\int_{0}^{t}\rho(\ta)\om(\ta)d\ta \big)$. \smallskip

\textbf{Small frequencies} \smallskip

In this case, if $\eta=\eta(t)$ is decreasing, in general all zones appear for small frequencies (see Fig. 1, Case a). Then, we have the following three cases: \smallskip

\textit{Case 1.1}:\, $t\leq t_{diss}$ \smallskip

In this case $(t,\xi)$ belongs to $\Zdiss(d_0)$. Then, we have the following estimate from Corollary \ref{Cor.Est.Diss.Zone}:
\[ \big( |E(t,0,\xi)| \big)\lesssim \frac{\la(t)}{F\big( \La(t) \big)}\left( \begin{array}{cc}
1 & 1 \\
1 & 1 \end{array} \right). \]

\textit{Case 1.2}:\, $t_{diss}\leq t\leq t_{|\xi|} $ \smallskip

Now we will glue the estimates in $\Zell(d_0,\ve)$ from Corollary \ref{Cor.Est.Ell.Zone} with those in $\Zdiss(d_0)$ from Corollary \ref{Cor.Est.Diss.Zone}.
\begin{lemma} \label{Par.Dep.Lemma.Case1.2}
The following estimates hold for all $t\in[t_{diss},t_{|\xi|}]$ and $|\xi|\geq \frac{d_0}{F(\La(t))}$:
\begin{align*}
\big( |E(t,0,\xi)| \big) \lesssim \exp \big( -C|\xi|^2B_\la(0,t) \big)
\left( \begin{array}{cc}
\la(t)|\xi| & \la(t)|\xi| \\
\frac{\la^2(t)|\xi|^2}{\rho(t)} & \frac{\la^2(t)|\xi|^2}{\rho(t)}
\end{array} \right).
\end{align*}
\end{lemma}
\begin{proof}
The fundamental solution $E=E(t,0,\xi)$ can be represented as
\[ E(t,0,\xi)=E_{ell}(t,t_{diss},\xi)E_{diss}(t_{diss},0,\xi). \]
Then, we have
\begin{align*}
\big( |E&(t,0,\xi)| \big)\lesssim \big( |E_{ell}(t,t_{diss},\xi)| \big) \big( |E_{diss}(t_{diss},0,\xi)| \big) \\
& \lesssim \exp \big( -C|\xi|^2B_\la(t_{diss},t) \big)\left( \begin{array}{cc}
\frac{\la(t)}{\la(t_{diss})} & \frac{\la(t)|\xi|}{\rho(t_{diss})} \\
\frac{\la^2(t)|\xi|}{\la(t_{diss})\rho(t)} & \frac{\la^2(t)|\xi|^2}{\rho(t_{diss})\rho(t)}
\end{array} \right) \frac{\la(t_{diss})}{F\big( \La(t_{diss}) \big)} \left( \begin{array}{cc}
1 & 1 \\
1 & 1
\end{array} \right) \\
& \lesssim \exp \big( -C|\xi|^2B_\la(0,t) \big)\left( \begin{array}{cc}
\la(t)|\xi| & \la(t)|\xi| \\
\frac{\la^2(t)|\xi|^2}{\rho(t)} & \frac{\la^2(t)|\xi|^2}{\rho(t)}
\end{array} \right),
\end{align*}
where we used $\la(t_{diss})|\xi|\lesssim \rho(t_{diss})$ and $|\xi|F\big( \La(t_{diss}) \big)=d_0$. Due to the first statement of Lemma \ref{Gluing.Auxiliary.Lemma} we can extend $B_\la(t_{diss},t)$ to $B_\la(0,t)$.
This completes the proof. \qed
\end{proof}

\textit{Case 1.3}:\, $ t\geq t_{|\xi|}$ \smallskip

To derive the corresponding estimates for $t\in [t_{|\xi|},\iy)$ we shall estimate the term
\[ S(t,|\xi|):=\exp \Big( -C|\xi|^2\int_{0}^{t_{|\xi|}}\frac{\la^2(\ta)\om(\ta)}{\rho(\ta)}d\ta \Big)\exp \Big( -\frac{1}{2}\int_{t_{|\xi|}}^{t}\rho(\ta)\om(\ta)d\ta \Big). \]
This term explains the competition of influences from different zones. We use the decreasing behavior of the function $S=S(t,|\xi|)$ in $|\xi|$. So, the function $S=S(t,|\xi|)$ takes its maximum for $\tilde{|\xi|}$ satisfying $t=t_{\tilde{|\xi|}}$, that is, the second integral vanishes in $S(t,|\xi|)$ (see \cite{BuiReiBook}).
\begin{lemma} \label{Glu.S(t,xi)}
For any $t\geq t_{|\xi|}$ and for a sufficiently small positive constant $C$ the function $S=S(t,|\xi|)$ satisfies the following estimate:
\[ S(t,|\xi|)\leq \max_{\xi\in \mathbb{R}^n}\Big\{ \exp \Big( -C|\xi|^2\int_{0}^{t}\frac{\la^2(\ta)\om(\ta)}{\rho(\ta)}d\ta \Big) \Big\}. \]	
\end{lemma}
Now we will glue the estimates from $\Pi_{hyp}(N,\ve)$ and $\Zell(d_0,\ve)$ for $|\xi|\geq\frac{d_0}{F(\La(t))}$.
\begin{lemma} \label{Par.Dep.Lemma.Case1.3}
The following estimates hold for all $|\xi|\geq\frac{d_0}{F(\La(t))}$ and $t\in[t_{|\xi|},\iy)$:
\begin{align*}
\big( |E(t,0,\xi)| \big) \lesssim \exp \big( -C'|\xi|^2B_\la(0,t) \big)\left( \begin{array}{cc}
\la(t) & \la(t)|\xi| \\
\la(t) & \la(t)|\xi|
\end{array} \right).
\end{align*}
\end{lemma}
\begin{proof}
Using the representation of the fundamental solution $E(t,0,\xi)$ as
\[ E(t,0,\xi)=E_{hyp}(t,t_{osc},\xi)E_{osc}(t_{osc},t_{red},\xi)E_{red}(t_{red},t_{ell},\xi)E_{ell}(t_{ell},0,\xi) \]
we arrive at the estimate
\begin{align*}
\big( |E&(t,0,\xi)| \big) \lesssim \big( |E_{red}(t,t_{|\xi|},\xi)| \big) \big( |E_{ell}(t_{|\xi|},0,\xi)| \big) \\
& \lesssim \Big( \frac{\de(t_{|\xi|})}{\de(t)} \Big)^{C_2}
\left( \begin{array}{cc}
1 & 1 \\
1 & 1
\end{array} \right)
\exp \big( -C_1|\xi|^2B_\la(0,t_{|\xi|}) \big)
\left( \begin{array}{cc}
\frac{\la(t_{|\xi|})}{\la(0)} & \frac{\la(t_{|\xi|})|\xi|}{\rho(0)} \\
\frac{\la^2(t_{|\xi|})|\xi|}{\la(0)\rho(t_{|\xi|})} & \frac{\la^2(t_{|\xi|})|\xi|^2}{\rho(0)\rho(t_{|\xi|})}
\end{array} \right) \\
& \lesssim \exp \big( -C'|\xi|^2B_\la(0,t) \big)\left( \begin{array}{cc}
\la(t) & \la(t)|\xi| \\
\la(t) & \la(t)|\xi|
\end{array} \right),
\end{align*}
where we used $\rho(t_{|\xi|})\approx\la(t_{|\xi|})|\xi|$. Defining $C':=\min\{C_1,C_2\}$ we used Lemma \ref{Glu.S(t,xi)} with
\[ \exp\big(-C_1|\xi|^2 B_\la(0,t_{|\xi|}) \big) \Big( \frac{\de(t_{|\xi|})}{\de(t)} \Big)^{C_2} \leq \exp\big( -C'|\xi|^2B_\la(0,t) \big). \]
This completes the proof. \qed
\end{proof}
Finally, for small frequencies it remains to glue the estimates in $\Pi_{hyp}(N,\ve)$ with those in $\Zell(d_0,\ve)$ and in $\Zdiss(d_0)$. We remark that this case comes into play only if $\eta=\eta(t)$ is decreasing. We have already obtained in Lemma \ref{Par.Dep.Lemma.Case1.2} the desired estimate after gluing the estimates in $\Zell(d_0,\ve)$ with those in $\Zdiss(d_0)$. Denoting the glued propagator by $\E=\E(t,0,\xi)$ we will only glue the estimates in $\Zred(\ve)$ with the estimate from Lemma \ref{Par.Dep.Lemma.Case1.2}.
\begin{lemma} \label{Par.Dep.Lemma.Case1.3}
The following estimates hold for all $t\in[t_{|\xi|},\iy)$:
\begin{align*}
\big( |E(t,0,\xi)| \big) \lesssim \exp \big( -C'|\xi|^2B_\la(0,t) \big) \lambda(t)|\xi|
\left( \begin{array}{cc}
1 & 1 \\
1 & 1
\end{array} \right).
\end{align*}
\end{lemma}
\begin{proof}
We have the following estimate in $\Pi_{hyp}(N,\ve)$:
\[ \big( |E_{red}(t,s,\xi)| \big)\lesssim \Big( \frac{\de(s)}{\de(t)} \Big)^{1-2\ve}
\left( \begin{array}{cc}
1 & 1 \\
1 & 1
\end{array} \right). \]
Then, by using the estimate from Lemma \ref{Par.Dep.Lemma.Case1.2} we get
\begin{align*}
\big( |E&(t,0,\xi)| \big) \lesssim \big( |E_{red}(t,t_{|\xi|},\xi)| \big) \big( |\E(t_{|\xi|},0,\xi)| \big) \\
& \lesssim \Big( \frac{\de(t_{|\xi|})}{\de(t)} \Big)^{C_2}
\left( \begin{array}{cc}
1 & 1 \\
1 & 1
\end{array} \right)
\exp \big( -C_1|\xi|^2B_\la(0,t_{|\xi|}) \big)
\left( \begin{array}{cc}
\la(t_{|\xi|})|\xi| & \la(t_{|\xi|})|\xi| \\
\frac{\la^2(t_{|\xi|})|\xi|^2}{\rho(t_{|\xi|})} & \frac{\la^2(t_{|\xi|})|\xi|^2}{\rho(t_{|\xi|})}
\end{array} \right) \\
& \lesssim \exp\big(-C'|\xi|^2 B_\la(0,t) \big)\la(t)|\xi|
\left( \begin{array}{cc}
1 & 1 \\
1 & 1
\end{array} \right)
\end{align*}
for all $t\geq t_{|\xi|}$. Here we used $\rho(t_{|\xi|})\approx \la(t_{|\xi|})|\xi|$. Defining $C':=\min\{C_1,C_2\}$ by Lemma \ref{Glu.S(t,xi)} we also used
\[ \exp\big(-C_1|\xi|^2 B_\la(0,t_{|\xi|}) \big)\Big( \frac{\de(t_{|\xi|})}{\de(t)} \Big)^{C_2}\leq \exp\big(-C'|\xi|^2 B_\la(0,t) \big). \]
This completes the proof. \qed
\end{proof}
\textbf{Case 2: the function $\eta=\eta(t)$ is monotonously increasing} \smallskip

In this case the elliptic zone lies on the top of the hyperbolic zone. Then, we have two different parts of the phase space to glue (see Fig. 1, Case b). \smallskip

\textbf{Small frequencies} \smallskip

Small frequencies lie completely inside $\Pi_{ell}(d_0,\ve)$. For this reason we can use the estimates that we obtained in \textit{Case 1.1} and \textit{Case 1.2}. \smallskip

\textbf{Large frequencies} \smallskip

\textit{Case 2.1}:\, $ t\leq t_{|\xi|}$ \smallskip

In this case $(t,\xi)$ belongs to $\Pi_{hyp}(N,\ve)$. Then, we have
\[ \big( |E(t,0,\xi)| \big)\lesssim \Big( \frac{1}{\de(t)} \Big)^{1-2\ve}
\left( \begin{array}{cc}
1 & 1 \\
1 & 1
\end{array} \right). \]

\textit{Case 2.2}:\, $ t\geq t_{|\xi|}$ \smallskip

First let us consider the case $(t,\xi)\in\Zell(d_0,\ve)$. Then, from Corollary \ref{Cor.Est.Ell.Zone} we get
\begin{align*}
\big( |E(t,0,\xi)| \big) \lesssim \exp \big( -C'|\xi|^2B_\la(0,t) \big)\left( \begin{array}{cc}
\frac{\la(t)}{\la(0)} & \frac{\la(t)|\xi|}{\rho(0)} \\
\frac{\la^2(t)|\xi|}{\la(0)\rho(t)} & \frac{\la^2(t)|\xi|^2}{\rho(0)\rho(t)}
\end{array} \right).
\end{align*}
Now we have to glue the estimates in $\Zell(d_0,\ve)$ with those in $\Pi_{hyp}(N,\ve)$.
\begin{lemma} \label{Par.Dep.Lemma.Case2.2}
The following estimates hold for all $t\in[t_{|\xi|},\iy)$:
\begin{align*}
\big( |E(t,0,\xi)| \big) \lesssim \exp \big( -C'|\xi|^2B_\la(0,t) \big) \frac{\lambda(t)}{\lambda(t_{|\xi|})}
\left( \begin{array}{cc}
1 & 1 \\
1 & 1
\end{array} \right).
\end{align*}
\end{lemma}
\begin{proof}
Taking account of the representation
\[ E(t,0,\xi)=E_{ell}(t,t_{ell},\xi)E_{red}(t_{ell},t_{red},\xi)E_{osc}(t_{red},t_{osc},\xi)E_{hyp}(t_{osc},0,\xi), \]
we have
\begin{align*}
\big( |E&(t,0,\xi)| \big)\lesssim \big( |E_{ell}(t,t_{|\xi|},\xi)| \big) \big( |E_{red}(t_{|\xi|},0,\xi)| \big) \\
& \lesssim \exp\big( -C_1|\xi|^2B_\la(t_{|\xi|},t) \big)
\left( \begin{array}{cc}
\frac{\la(t)}{\la(t_{|\xi|})} & \frac{\la(t)|\xi|}{\rho(t_{|\xi|})} \\
\frac{\la^2(t)|\xi|}{\la(t_{|\xi|})\rho(t)} & \frac{\la^2(t)|\xi|^2}{\rho(t_{|\xi|})\rho(t)}
\end{array} \right)\Big( \frac{1}{\de(t_{|\xi|})} \Big)^{C_2}
\left( \begin{array}{cc}
1 & 1 \\
1 & 1
\end{array} \right) \\
& \lesssim \exp \big( -C'|\xi|^2B_\la(0,t) \big)
\left( \begin{array}{cc}
\frac{\la(t)}{\la(t_{|\xi|})} & \frac{\la(t)}{\la(t_{|\xi|})} \\
\frac{\la(t)}{\la(t_{|\xi|})} & \frac{\la(t)}{\la(t_{|\xi|})}
\end{array} \right),
\end{align*}
where we used $\rho(t_{|\xi|})\approx\la(t_{|\xi|})|\xi|$. Defining $C':=\min\{C_1,C_2\}$ and Lemma \ref{Par.Dep.Lemma.Case1.3}
we also used
\begin{equation*}
\exp\big(-C_1|\xi|^2 B_\la(t_{|\xi|},t)\big)\Big( \frac{1}{\de(t_{|\xi|})} \Big)^{C_2}\leq \exp\big( -C'|\xi|^2 B_\la(0,t) \big).
\end{equation*}
This completes the proof. \qed
\end{proof}
\subsection{Preliminaries}
Now we introduce $\hat{K}_0=\hat{K}_0(t,0,\xi)$ as the solution of the Cauchy problem \eqref{Eq.first.Fourier} with the initial conditions $\hat{u}_0(\xi)=1$ and $\hat{u}_1(\xi)=0$. Then, we have the following identity for $k=1,2$:
\begin{align*}
\left( \begin{array}{cc}
\frac{\la(t)|\xi|}{h_k(t,\xi)} & 0 \\
0 & 1
\end{array} \right)E_k(t,0,\xi) \left( \begin{array}{cc}
h_k(0,\xi) \\ 0
\end{array} \right) & = \left( \begin{array}{cc}
\frac{\la(t)|\xi|}{h_k(t,\xi)} & 0 \\
0 & 1
\end{array} \right) \left( \begin{array}{cc}
h_k(t,\xi)\hat{K}_0(t,0,\xi) \\
D_t\hat{K}_0(t,0,\xi)
\end{array} \right) \\
& =\left( \begin{array}{cc}
\la(t)|\xi|\hat{K}_0(t,0,\xi) \\
D_t\hat{K}_0(t,0,\xi)
\end{array} \right).
\end{align*}
Moreover, it holds
\begin{align*}
\left( \begin{array}{cc}
\frac{\la(t)|\xi|}{h_k(t,\xi)} & 0 \\
0 & 1
\end{array} \right)E_k(t,0,\xi)\left( \begin{array}{cc}
h_k(0,\xi) \\ 0
\end{array} \right) &= \left( \begin{array}{cc}
\frac{\la(t)|\xi|}{h_k(t,\xi)} & 0 \\
0 & 1
\end{array} \right) \left( \begin{array}{cc}
h_k(0,\xi)E_k^{(11)}(t,0,\xi) \\
h_k(0,\xi)E_k^{(21)}(t,0,\xi)
\end{array} \right) \\
& = \left( \begin{array}{cc}
\frac{h_k(0,\xi)}{h_k(t,\xi)}\la(t)|\xi|E_k^{(11)}(t,0,\xi) \\
h_k(0,\xi)E_k^{(21)}(t,0,\xi)
\end{array} \right),
\end{align*}
where $h_1=h_1(t,\xi)$ and $h_2=h_2(t,\xi)$ are defined in \eqref{h_1(t,xi)} and \eqref{h_2(t,xi)}, respectively. Moreover, $E_1(t,0,\xi):=E(t,0,\xi)$ and $E_2(t,0,\xi):=E_V(t,0,\xi)$ are defined in Definition \ref{Def.Fund.Sol.}. The above relations allows us to transfer properties of $E_k=E_k(t,0,\xi)$ to $\hat{K}_0=\hat{K}_0(t,0,\xi)$ for $k=1,2$. Thus, we obtain
\begin{align*}
\hat{K}_0(t,0,\xi) &= \frac{h_k(0,\xi)}{h_k(t,\xi)}E_k^{(11)}(t,0,\xi), \\
D_t\hat{K}_0(t,0,\xi) &= h_k(0,\xi)E_k^{(21)}(t,0,\xi).
\end{align*}
In the same way, we consider $\hat{K}_1=\hat{K}_1(t,0,\xi)$ as the solution of the Cauchy problem \eqref{Eq.first.Fourier} with initial conditions $\hat{u}_0(\xi)=0$ and $\hat{u}_1(\xi)=1$. Then, we get
\begin{align*}
\hat{K}_1(t,0,\xi) &= \frac{1}{h_k(t,\xi)}E_k^{(12)}(t,0,\xi), \\
D_t\hat{K}_1(t,0,\xi) &= E_k^{(22)}(t,0,\xi).
\end{align*}
\subsection{Final estimates}
Let us define for any $t\geq 0$ the function
\[ \Om(0,t):=\max \big\{\eta(0),\eta(t) \big\}\sqrt{1-\ve^2}. \]
\begin{remark} We distinguish between small and large frequencies. Small frequencies satisfy the condition $|\xi|\leq\Om(0,t)$, while, large frequencies satisfy the condition $|\xi|\geq \Om(0,t)$.
\end{remark}
Summarizing, we arrive at the following estimates for  $\big| \pa_t^l\hat{K}_j(t,0,\xi) \big|$ with $j,l=0,1$.
\begin{corollary}
	If $|\xi|\geq \Om(0,t)$, then we have the following estimates for $j,l=0,1$: 	
	\begin{align}
	\label{Par.Dep.Final.Large.1} \big| \pa_t^l\hat{K}_j(t,0,\xi) \big| \lesssim \la(t)^{l-1}|\xi|^{l-j}\Big( \frac{1}{\de(t)} \Big)^{1-2\ve}.
	\end{align}
	If $\dfrac{d_0}{F( \La(t))}\leq|\xi|\leq \Om(0,t)$, then we have the following estimates $j,l=0,1$:
	\begin{align}
	\label{Par.Dep.Final.Small.2} \big| \pa_t^l\hat{K}_j(t,0,\xi) \big| \lesssim \la^{l}(t)|\xi|^{l}\exp\big( -C|\xi|^2 B_\la(0,t) \big).
	\end{align}
    If $|\xi|\leq \dfrac{d_0}{F( \La(t))}$, then we have the following estimates $j,l=0,1$:
	\begin{align}
	\label{Par.Dep.Final.Diss.2} \big| \pa_t^l\hat{K}_j(t,0,\xi) \big| \lesssim \frac{\la^l(t)}{F^l(\La(t))}.
	\end{align}
\end{corollary}
\section{Energy estimates for solutions to damped wave models with additional regularity of the data}
The solution $u=u(t,x)$ to the Cauchy problem
\begin{align} \label{Par.Dep.Cauchy.Prob.Mats.}
\begin{cases}
u_{tt}-\la^2(t)\om^2(t)\lap u +\rho(t)\om(t)u_t=0, & (t,x)\in[0,\infty)\times \mathbb{R}^n, \\
u(0,x)=u_0(x), \,\,\,\, u_t(0,x)=u_1(x), & x\in\mathbb{R}^n,
\end{cases}
\end{align}
can be represented as
\[ u(t,x)=K_0(t,0,x)\ast_{(x)}u_0(x)+K_1(t,0,x)\ast_{(x)}u_1(x). \]
Thus, we may conclude the following estimate for the solution $u=u(t,x)$:
\[ \|u(t,\cdot)\|_{L^2}=\|\fu(t,\cdot)\|_{L^2}\leq \big\| \hat{K}_0(t,0,\xi)\fu_0(\xi) \big\|_{L^2}+\big\| \hat{K}_1(t,0,\xi)\fu_1(\xi) \big\|_{L^2}. \]
In order to estimate the $L^2$ norm of $\pa_t^l\pa_x^{\si}K_j(t,0,x)\ast_{(x)}u_j(x)$ for any $\si\geq 0$ and $j,l=0,1$, we can follow the techniques used in \cite{Dab.Luc.Rei.} and \cite{Matsumura76}. Then, we have the following statements for large and small frequencies.
\begin{lemma} \label{Par.Dep.Lemma.Large.Fr.}
    The following estimates hold for large frequencies $|\xi|\geq\Om(0,t)$:
    \begin{equation} \label{Par.Dep.Matsumura.Large.Freq.}
    \big\| |\xi|^{\si}\pa_t^l\hat{K}_j(t,0,\cdot)\fu_j \big\|_{L^2_{\textstyle\{ |\xi|\geq\Om(0,t) \}}} \lesssim \la^l(t)\Big( \frac{1}{\de(t)} \Big)^{1-2\ve}\|u_j\|_{H^{|\si|+l-j}}
    \end{equation}
    with $\si+l\geq j$ and for any $\si\geq 0$ and $j,l=0,1$. Moreover, if $\si=l=0$ and $j=1$ we distinguish between the following two cases:
    \begin{enumerate}
    \item If $\eta=\eta(t)$ is increasing, then we have the estimate
    \begin{equation} \label{Par.Dep.Matsumura.Large.Freq.all0}
    \big\| \hat{K}_1(t,0,\cdot)\fu_1 \big\|_{L^2_{\textstyle\{ |\xi|\geq\Om(0,t) \}}} \lesssim \frac{1}{\eta(t)} \Big( \frac{1}{\de(t)} \Big)^{1-2\ve}\|u_1\|_{L^2}.
    \end{equation}
    \item If $\eta=\eta(t)$ is decreasing we have the estimate
    \begin{equation} \label{decresingetaall0}
    \big\| \hat{K}_1(t,0,\cdot)\fu_1 \big\|_{L^2_{\textstyle\{ |\xi|\geq\Om(0,t) \}}} \lesssim \Big( \frac{1}{\de(t)} \Big)^{1-2\ve} \|u_1\|_{L^2}.
    \end{equation}
    \end{enumerate}
\end{lemma}
\begin{proof}
We have the following estimate for $\sigma+l\geq j$:
\begin{align*}
\big\| |\xi|^{\sigma}\pa_t^l&\hat{K}_j(t,0,\cdot)\fu_j \big\|_{L^2_{\textstyle\{ |\xi|\geq\Om(0,t) \}}} \\
& \leq\big\| |\xi|^{j-l}\pa_t^l\hat{K}_j(t,0,\xi) \big\|_{L^\iy_{\textstyle\{ |\xi|\geq\Om(0,t) \}}} \big\| |\xi|^{\sigma+l-j}\fu_j \big\|_{L^2_{\textstyle\{ |\xi|\geq\Om(0,t) \}}}.
\end{align*}
The second term on the right-hand side can be estimated by $\|u_j\|_{H^{\sigma+l-j}}$. Now let us consider the $L^\iy$ norm of $\pa_t^lK_j(t,0,\xi)$. Indeed, by using the estimate \eqref{Par.Dep.Final.Large.1} we get
\[ |\xi|^{j-l}\big| \pa_t^l\hat{K}_j(t,0,\xi) \big| \lesssim \la(t)^{l-1}\Big( \frac{1}{\de(t)} \Big)^{1-2\ve} \lesssim \la^l(t)\Big( \frac{1}{\de(t)} \Big)^{1-2\ve}. \]
Let $\sigma=l=0$ and $j=1$. Then, we arrive at the estimate
\[ \big\| \hat{K}_1(t,0,\cdot)\fu_1 \big\|_{L^2_{\textstyle\{ |\xi|\geq\Om(0,t) \}}} \lesssim \frac{1}{\Omega(0,t)}\Big( \frac{1}{\de(t)} \Big)^{1-2\ve}\|u_1\|_{L^2}. \]
This completes the proof. \qed
\end{proof}
\begin{lemma} \label{Par.Dep.Lemma.Small.Fr.(All)}
The following estimates hold for small frequencies $\frac{d_0}{F(\Lambda(t))}\leq |\xi|\leq \Om(0,t)$:
\begin{align} \label{Par.Dep.Matsumura.Small.Freq.} \nonumber
\big\| |\xi|^{\sigma}\pa_t^l\hat{K}_j(t,0,\cdot&)\fu_j \big\|_{L^2_{\textstyle\{\frac{d_0}{F(\Lambda(t))}\leq|\xi|\leq\Om(0,t)\}}} \\
& \lesssim \la^{l}(t)\big( B_\la(0,t) \big)^{-\frac{l}{2}} \big( B_\la(0,t) \big)^{-\frac{\sigma}{2}-\frac{n}{2}( \frac{1}{m}-\frac{1}{2})} \|u_j\|_{L^m}
\end{align}
for any $\sigma\geq 0$ and $j,l=0,1$, where $m \in [1,2)$.
\end{lemma}
\begin{lemma} \label{Par.Dep.Lemma.Small.Fr.(Diss)}
The following estimates hold for small frequencies $|\xi|\leq\frac{d_0}{F(\La(t))}$:
\begin{align} \label{Par.Dep.Matsumura.Diss.Zone} \nonumber
\big\| |\xi|^{\sigma}\pa_t^l\hat{K}_j(t,0,\cdot&)\fu_j \big\|_{L^2_{ \textstyle\{|\xi|\leq\frac{d_0}{F(\La(t))}\} }} \\
& \lesssim \la^l(t) \big( F^2(\La(t)) \big)^{-\frac{l}{2}} \big( F^2(\La(t)) \big)^{-\frac{\sigma}{2}-\frac{n}{2}(\frac{1}{m}-\frac{1}{2})} \|u_j\|_{L^m}
\end{align}
for any $\sigma\geq 0$ and $j,l=0,1$, where $m\in [1,2)$.
\end{lemma}
\begin{proof}The proofs of Lemmas \ref{Par.Dep.Lemma.Small.Fr.(All)} and \ref{Par.Dep.Lemma.Small.Fr.(Diss)} use H\"older's inequality and Hausdorff-Young inequality. \qed  \end{proof}
\subsubsection{Proof of Theorem \ref{Par.Dep.Final.Thm.with.s=0}.}
Due to condition (B6), from (\ref{Par.Dep.Matsumura.Small.Freq.}) and (\ref{Par.Dep.Matsumura.Diss.Zone}) it follows
\[ \big( F^2(\La(t)) \big)^{-\frac{\sigma}{2}-\frac{n}{2}(\frac{1}{m}-\frac{1}{2})-\frac{l}{2}} \lesssim \big( B_\la(0,t) \big)^{-\frac{\sigma}{2}-\frac{n}{2}( \frac{1}{m}-\frac{1}{2})-\frac{l}{2}} \]
for $l=0,1$. For this reason the right-hand sides in the estimates of Lemmas \ref{Par.Dep.Lemma.Large.Fr.} and \ref{Par.Dep.Lemma.Small.Fr.(Diss)} decay faster than that one in the estimates of Lemma \ref{Par.Dep.Lemma.Small.Fr.(All)}.
On the other hand, the regularity of the data is coming from the large frequencies from the estimates
(\ref{Par.Dep.Matsumura.Large.Freq.}) to (\ref{decresingetaall0}).
In this way we have proved the desired statements. \qed
\begin{remark}
We note that in our estimates we replace $B_\la(0,t)$ by $1+B_\la(0,t)$. This can be done modulo a compact set in the extended phase space. Such a compact set will never influence the desired estimates.
\end{remark}
\section{Some examples}
\begin{example}[\textit{Polynomial case}] \label{Example.Polynomial}
Let $\la(t)=(\alpha+1)(1+t)^{\alpha}$, $\alpha>0$. So, we have
\[ \La(t)=(1+t)^{\alpha+1} \,\,\, \mbox{and} \,\,\, \Th(t)=(1+t)^{\ga+1}, \,\,\, -1<\ga<\alpha. \]
Moreover, we choose
\[ \rho(t)=\frac{(\alpha+1)^2}{2\alpha-\ba+1}(1+t)^\ba, \,\,\, \alpha-\gamma-1<\ba<2\alpha+1. \]
If we take $\Xi(t)=(1+t)^\kappa$ from the condition (A2), then the condition (A4) holds with
\[ 1>\kappa\geq \kappa_M=1-\alpha+\ga+\frac{\alpha-\ga}{M+1}, \]
where $M$ denotes the order of regularity for the coefficients. On the other hand, by the condition (A5) we obtain
\[ F(\La(t)) \approx (1+t)^{\alpha+2\kappa-1}, \,\,\, \alpha+2\kappa-1>0. \]
Finally, from the condition (B6) we have the estimate
\[ (1+t)^{2\alpha-\beta+1} \leq (1+t)^{2\alpha+4\kappa-2}, \,\,\, \kappa \geq \frac{3-\beta}{4}.\]
Summarizing, the above relations yield
\[ \kappa \geq \frac{3-\beta}{4} > \frac{1-\alpha}{2}>1-\alpha+\gamma + \frac{\alpha-\gamma}{M+1}.\]
Therefore, we obtain that the range for admissible $\kappa$ is given by $\kappa \geq \frac{3-\beta}{4}$. Hence, the hypotheses of Theorem \ref{Par.Dep.Final.Thm.with.s=0} are satisfied. \\
Let us introduce $C_{\sigma}=\frac{\sigma}{2}+\frac{n}{2}\big( \frac{1}{m}-\frac{1}{2} \big)$. Then, we have the following estimates:
\begin{align*}
\|u(t,\cdot)\|_{\dot{H}^{\sigma}} & \lesssim (1+t)^{-(2\alpha-\beta+1)C_{\sigma}} \big( \|u_0\|_{L^m\cap H^{\sigma}}+\|u_1\|_{L^m\cap H^{[\sigma-1]_+}} \big), \\
\|u_t(t,\cdot)\|_{\dot{H}^{\sigma}} & \lesssim (1+t)^{-(2\alpha-\beta+1)C_{\sigma}+\frac{\beta-1}{2}} \big( \|u_0\|_{L^m\cap H^{\sigma+1}}+\|u_1\|_{L^m\cap H^{\sigma}} \big).
\end{align*}
\end{example}
\begin{example}[\textit{Exponential case}] \label{Example.Exponential}
Let us choose $\la(t)=e^t$. So, we have
\[ \La(t)=e^t \,\,\, \mbox{and} \,\,\, \Th(t)=e^{rt}, \,\,\, 0<r<1.\]
Moreover, we choose
\[ \rho(t)=\frac{1}{2-q}e^{qt}, \,\,\, 1-r<q<2.  \]
Now if we take $\Xi(t)=e^{\kappa t}$ from the condition (A2), then the condition (A4) is satisfied with
\[ 0>\kappa\geq \kappa_M=r-1+\frac{1-r}{M+1}. \]
On the other hand, by the condition (A5) we get
\[ F(\La(t)) \approx e^{(2\kappa+1)t}, \,\,\, \kappa \geq -\frac{q}{4}. \]
Finally, by the condition (B6) we have the estimate
\[ e^{(2-q)t}\leq e^{(2+4\kappa)t}, \,\,\, \kappa\geq-\frac{q}{4}. \]
The above relations yield
\[ \kappa \geq -\frac{q}{4} >-\frac{1}{2}> r-1+\frac{1-r}{M+1}.\]
Therefore, we obtain that the range for admissible $\kappa$ is given by $\kappa\geq-\frac{q}{4}$. Hence, the hypotheses of Theorem \ref{Par.Dep.Final.Thm.with.s=0} are satisfied. \\
Introducing $C_{\sigma}=\frac{\sigma}{2}+\frac{n}{2}\big( \frac{1}{m}-\frac{1}{2} \big)$, we have the following estimates:
\begin{align*}
\|u(t,\cdot)\|_{\dot{H}^{\sigma}} & \lesssim  e^{-(2-q)C_{\sigma}t} \big( \|u_0\|_{L^m\cap H^{\sigma}}+\|u_1\|_{L^m\cap H^{[\sigma-1]_+}} \big), \\
\|u_t(t,\cdot)\|_{\dot{H}^{\sigma}} & \lesssim e^{-(2-q)C_{\sigma}t}e^{\frac{q}{2}t} \big( \|u_0\|_{L^m\cap H^{\sigma+1}}+\|u_1\|_{L^m\cap H^{\sigma}} \big).
\end{align*}
\end{example}
\begin{example}[\textit{Super-exponential case}] \label{Example.Super-exponential}
Let us choose $\la(t)=e^te^{e^t}$. So, we have
\[ \La(t)=e^{e^t} \,\,\, \mbox{ and} \,\,\, \Th(t)=e^{re^t}, \,\,\, 0<r<1. \]
Moreover, we choose
\[\rho(t)=\frac{1}{2-q}e^te^{qe^t}, \,\,\, 1-r<q<2.\]
Finally, we take $\Xi(t)=e^{-t}e^{\kappa e^t}$ from condition (A2). Then, the condition (A4) holds with
\[ 0>\kappa\geq \kappa_M=r-1+\frac{1-r}{M+1}. \]
On the other hand, by condition (A5) we have
\[ F(\La(t)) \approx e^{(1+2\kappa)e^t}, \,\,\, 1+2\kappa >0. \]
From condition (B6) we have the estimate
\[ e^{(2-q)e^t} \leq e^{(2+4\kappa)e^t}, \,\,\, \kappa \geq -\frac{q}{4}.\]
For this reason we obtain that the range of admissible $\kappa$ is given by $\kappa \geq -\frac{q}{4}$.
Summarizing, the above relations yield
\[ \kappa \geq -\frac{q}{4} > -\frac{1}{2} > r-1+\frac{1-r}{M+1}.\]
Hence, the hypotheses of Theorem \ref{Par.Dep.Final.Thm.with.s=0} are satisfied. \\
Introducing $C_{\sigma}=\frac{\sigma}{2}+\frac{n}{2}\big( \frac{1}{m}-\frac{1}{2} \big)$, we get the following estimates:
\begin{align*}
\|u(t,\cdot)\|_{\dot{H}^{\sigma}} & \lesssim  e^{-(2-q)C_{\sigma}e^t} \big( \|u_0\|_{L^m\cap H^{\sigma}}+\|u_1\|_{L^m\cap H^{[\sigma-1]_+}} \big), \\
\|u_t(t,\cdot)\|_{\dot{H}^{\sigma}} & \lesssim e^te^{-(2-q)C_{\sigma}e^t}e^{\frac{q}{2}e^t} \big( \|u_0\|_{L^m\cap H^{\sigma+1}}+\|u_1\|_{L^m\cap H^{\sigma}} \big).
\end{align*}
\end{example}
\subsubsection{Construction of an admissible oscillating function.}
Now we present an admissible non-trivial oscillating function $\om=\om(t)$ in \eqref{Par.Dep.Cauchy.Prob.Mats.} satisfying the hypotheses of Theorem \ref{Par.Dep.Final.Thm.with.s=0}. The construction of the function $\om=\om(t)$ was proposed in \cite{H-W09}. \\
In order to construct a non-trivial function $\om=\om(t)$ satisfying (A2), (A3) and (B5) let us choose the positive sequences $\{t_j\}_j$, $\{\de_j\}_j$ and $\{\eta_j\}_j$ as follows:
\begin{equation} \label{Const.om(t).seq.}
t_j\ra\iy, \quad \de_j\leq \Delta t_j:=t_{j+1}-t_j \quad \text{and} \quad \eta_j\leq 1,
\end{equation}
and a function $\psi\in C_0^M(\mathbb{R})$ with
\[ \supp\psi\subseteq[0,1], \quad -1<\psi(t)<1 \quad \text{and} \quad \int_{0}^{1}|\psi(t)|dt=\frac{1}{2}.  \]
Then, we define
\[ \om(t)=1+\sum_{j=1}^{\iy}\eta_j\psi\Big( \frac{t-t_j}{\de_j} \Big). \]
The last sum is convergent, since by \eqref{Const.om(t).seq.} for each $t$ at most one term is present. Furthermore, if $c_0=\min\psi(t)$ and $c_1=\max\psi(t)$, then we get the bounds
\[ 0<1+c_0\leq \om(t)\leq 1+c_1. \]
For (A2) we can take
\begin{equation} \label{Const.om(t).Xi(t)}
C^{-1}_1\Xi(t_j)\leq \de_j\leq C_1\Xi(t_j)
\end{equation}
and the sequence $\{t_j\}_j$ satisfying
\begin{equation} \label{Const.om(t).la,La}
C_2^{-1}\la(t_{k+1})\leq \la(t_k)\leq C_2\la(t_{k+1}) \,\,\, \text{and} \,\,\, C_3^{-1}\La(t_{k+1})\leq \La(t_k)\leq C_3\La(t_{k+1})
\end{equation}
with positive constants $C_j$, $j=1,2,3$, which are independent of $k$. Indeed, by \eqref{Const.om(t).seq.} and the definition of $\om=\om(t)$ we have
\[ \om(t)=1+\eta_k\psi\Big( \frac{t-t_k}{\de_k} \Big) \quad \text{for all} \quad t\in[t_k,t_{k+1}]. \]
Taking account of \eqref{Const.om(t).Xi(t)} and \eqref{Const.om(t).la,La} it follows
\[ |d_t^i\om(t)|\leq C_i\frac{\eta_k}{\de_k^i}\leq \tilde{C}_i\Xi^{-i}(t) \quad \text{for all} \,\,\, t\in[t_k,t_{k+1}], \,\,\, i=1,\cdots,M. \]
Using the definition of $\om=\om(t)$ we may conclude for all $t\in[t_k,t_{k+1}]$ the estimate
\[ \int_{0}^{t}\la(s)|\om(s)-1|ds = \sum_{j=1}^{k}\eta_j\int_{t_j}^{t_{j+1}}\la(s)\Big|\psi\Big( \frac{s-t_j}{\de_j} \Big)\Big|ds\leq\frac{1}{2}\sum_{j=1}^{k}\eta_j\de_j\la(t_{j+1}). \]
This implies that the stabilization condition (A3) is ensured if we assume $\eta_j\de_j$ are small enough. Indeed,
\[ \Th(t_{k+1})\approx \sum_{j=1}^{k}\eta_j\de_j\la(t_{j+1}) = o\Big( \sum_{j=1}^{k}\la(t_j)\Delta(t_j) \Big). \]
\begin{example}[\textit{Polynomial case}]
We consider $\la(t)=(\al+1)(1+t)^\alpha$, $\alpha>0$. To define one admissible function $\om=\om(t)$ let us choose the parameters $\alpha$, $\ga$ and $\kappa$ from Example \ref{Example.Polynomial} and positive sequences $\{t_j\}_j$, $\{\de_j\}_j$ and $\{\eta_j\}_j$ as follows:
\[ t_j=2^j, \quad \de_j=2^{\kappa j}\leq \Delta t_j=t_{j+1}-t_j=2^j \quad \text{and} \quad \eta_j=2^{j(\ga-\alpha-\kappa)}. \]
From Example \ref{Example.Polynomial} we have
\[ 1>\kappa\geq\frac{3-\ba}{4}>\kappa_M=1-\al+\gamma+\frac{\al-\ga}{M+1}. \]
Therefore, we get $\ga-\al-\kappa<0$ and this implies that we have $0<\eta_j\leq1$. Moreover, the stabilization condition (A3) holds, since $\eta_j\delta_j=2^{j(\ga-\al)}<1$, where $\ga-\al<0$.
\end{example}
\begin{example}[\textit{Exponential case}]
We consider $\la(t)=e^t$. Let us choose the parameters $r$ and $\kappa$ from Example \ref{Example.Exponential} and the positive sequences $\{t_j\}_j$, $\{\de_j\}_j$ and $\{\eta_j\}_j$ by
\[ t_j=j, \quad \de_j=e^{\kappa j} \leq \Delta t_j=t_{j+1}-t_j \quad \text{and} \quad \eta_j=e^{j(r-\kappa-1)}. \]
From Example \ref{Example.Exponential} we have
\[ 0>\kappa\geq-\frac{q}{4}>\kappa_M=r-1+\frac{1-r}{M+1}. \]
Therefore, we have $\kappa<0$ and $r-\kappa-1<0$. These imply that we have $\de_j<1$ and $0<\eta_j\leq1$, respectively. Moreover, the stabilization condition (A3) is satisfied, because $\eta_j\delta_j=e^{j(r-1)}<1$, where $r-1<0$.
\end{example}
\begin{example}[\textit{Super-exponential case}]
We consider $\la(t)=e^te^{e^t}$. Let us choose the parameters $r$ and $\kappa$ from Example \ref{Example.Super-exponential} and positive sequences $\{t_j\}_j$, $\{\de_j\}_j$ and $\{\eta_j\}_j$ as follows:
\[ t_j=e^j, \quad \de_j=e^{-j}e^{\kappa e^j}\leq \Delta t_j=t_{j+1}-t_j \quad \text{and} \quad \eta_j=e^{(r-\kappa-1)e^j}. \]
From Example \ref{Example.Super-exponential} we have
\[ 0>\kappa\geq-\frac{q}{4}>\kappa_M=r-1+\frac{1-r}{M+1}. \]
These imply that we have $\de_j<1$ and $0<\eta_j\leq1$, respectively. Moreover, the stabilization condition (A3) holds, since $\eta_j\delta_j=e^{j(r-1)}<1$.
\end{example}
\section{Concluding remarks and open problems}
\begin{remark}
The recent papers \cite{BuiReisup.} and \cite{BuiReisub.} are devoted to studying the following Cauchy problem for semi-linear wave models with effective dissipation:
\begin{align} \label{Semi-linear.Cauchy.Bui}
\begin{cases}
u_{tt}-a^2(t)\Delta u+b(t)u_t=|u|^p, & (t,x)\in[0,\infty)\times \mathbb{R}^n, \\
u(0,x)=u_0(x), \,\,\,\, u_t(0,x)=u_1(x), & x\in\mathbb{R}^n.
\end{cases}
\end{align}
Here the considerations are divided into two cases depending on the behavior of the propagation speed: \textit{the super-exponential case} and \textit{the sub-exponential case}, respectively, for the global (in time) existence of small data solutions to \eqref{Semi-linear.Cauchy.Bui}.
\end{remark}
\begin{remark}
An interesting application of the results of this paper is to study the following Cauchy problem for semi-linear damped wave models with time-dependent speed of propagation and ``effective-like'' damping term $\rho(t)\om(t)u_t$ in combination with very-fast oscillations and stabilization condition:
\begin{align*}
\begin{cases}
u_{tt}-\la^2(t)\om^2(t)\Delta u+\rho(t)\om(t)u_t=|u|^p, & (t,x)\in[0,\infty)\times \mathbb{R}^n, \\
u(0,x)=u_0(x), \,\,\,\, u_t(0,x)=u_1(x), & x\in\mathbb{R}^n.
\end{cases}
\end{align*}
This will be discussed in a forthcoming paper to understand the influence of the oscillations on the global (in time) existence of small data solutions. The key tools are the Gagliardo-Nirenberg inequalities, the fractional chain rule, the fractional Leibniz rule and the fractional powers rules which have been extensively discussed in Harmonic Analysis (cf. the book \cite{Ebert.Reissig.Book}).
\end{remark}

\end{document}